\documentclass[reqno,11pt]{amsart}
\usepackage[T1]{fontenc}
\usepackage[utf8]{inputenc}
\usepackage{amsmath,amsfonts,amssymb}
\usepackage{fullpage}
\usepackage{latexsym,color}
\usepackage{lmodern}
\usepackage{enumitem}

\usepackage{graphicx}
\graphicspath{{./figures_JCA/}}

\usepackage{tikz} 
\usetikzlibrary{calc,spy,patterns} 

\usepackage{dsfont}
\usepackage{microtype} 
\usepackage{float}

\usepackage{hyperref} 
\hypersetup{colorlinks = true, linkcolor = blue, citecolor=red}

\newtheorem{theorem}{Theorem}[section]
\newtheorem{proposition}[theorem]{Proposition}
\newtheorem{lemma}[theorem]{Lemma}
\newtheorem{corollary}[theorem]{Corollary}

\theoremstyle{definition}
\newtheorem{definition}[theorem]{Definition}
\newtheorem{example}[theorem]{Example}

\newtheorem{remark}[theorem]{Remark}

\numberwithin{equation}{section}

\def\e{\varepsilon}
\def\ov{\overline}

\DeclareMathOperator{\Div}{div}
\renewcommand{\div}{\Div}

\DeclareMathOperator{\conv}{Conv}

\DeclareMathOperator{\trace}{Tr}

\DeclareMathOperator{\cl}{cl}

\newcommand{\LL}{\mathcal{L}}
\newcommand{\HH}{\mathcal{H}}

\newcommand{\DD}{\mathcal{D}}
\newcommand{\RR}{\mathbb{R}}
\def\R{\mathbb{R}}

\newcommand{\minus}{\!-\!}

\newcommand{\tinm}[1]{\text{\;#1\;}}
\newcommand{\ps}[1]{\langle #1\rangle}

\newcommand{\1}{{{\bf 1}\kern-0,28em\rm l}}
\newcommand{\one}{{{\bf 1}\kern-0,28em\rm l}} 
\newcommand{\med}{\medskip\noindent}

\newcommand{\f}{\varphi}

\def\sepa{\; : \;}

\def\sto{\overset{*}{\rightharpoonup}}

\def\n{{\bf n}}
\def\proj{\Pi}

\newcommand{\aco}[1]{\left\{#1\right\}}
\newcommand{\pare}[1]{\left(#1\right)}

\def\res{\text{\huge$\llcorner$}}
\def\a{\alpha}

\def\O{\Omega}
\def\l{\lambda}

\date{}

\begin{document}

\title[Calibrations for free boundary and Cheeger-type problems] {Calibrations  for minimal surfaces with free boundary and Cheeger-type problems}

\medskip

\dedicatory{ In memory of Hedy Attouch, his mathematical elegance, and his valuable friendship over the years. }

\author[G.~Bouchitt\'e, Minh.~Phan]{Guy Bouchitt\'e\and  Minh Phan}
\address[Guy Bouchitt\'e]{UFR des Sciences et Techniques\\ Universit\'e de Toulon et du Var, BP 132\\ 83957 La Gard Cedex (France)
}
\email{bouchitte@univ-tln.fr}

\address[Minh  Phan]{ The University of Danang - University of Science and Education, Da Nang 550000, (Viet Nam)}
\email{ptdminh@ued.udn.vn}

\begin{abstract}
We study a problem of minimal surfaces with free boundary  written in  the form of a non convex minimization problem. 
Our aim is to characterize  optimal solutions by finding a suitable calibration field. A natural upper bound of the infimum is given by a variant of the Cheeger problem that we solve explicitly proving the optimality  thanks to the construction of a cut-locus potential. The comparison with the original problem is then discussed in detail.
\end{abstract}

\maketitle



\section{Introduction}

In all the paper, $D$ is a given  bounded  domain of $\RR^N$ with Lipschitz boundary. For any subset $\Omega\subset D$, we denote by $|\Omega|:=\LL^N(\Omega)$ its area (the $N$-dimensional Lebesgue measure of $\Omega$) and by $P(\Omega)$ its perimeter. Recall that, if $P(\Omega)<+\infty$, then $P(\Omega)=\HH^{N-1}(\partial \Omega)$, i.e. the $(N-1)$-dimensional Hausdorff  measure of the (essential) boundary of $\Omega$.

\medskip
Given   $\lambda \ge 0$,   our aim is to study  the following variational problem: 
\begin{align}\label{pbbeta}
\beta(\lambda)&:=\inf \left\{ \int_D \sqrt{1+|\nabla u|^2}\; dx -  \lambda\;|\{u \ge 1\}|\; : \; u\in W^{1,1}_0(D) \right\}.
\end{align}
By using a truncation argument, one checks easily that the infimum above is unchanged if we restrict the infimum to competitors $u$
such that $0\le u\le 1$.  The free boundary associated with such $u$ is then the essential boundary of the set $\Omega=\Omega(u) := \{ u = 1\}$.  Accordingly, the minimization problem \eqref{pbbeta} contains two terms in competition, on the one hand the minimal area of a parametrized surface with $\partial D\times \{0\}$ and $\partial \O \times \{1\}$ boundaries, and on the other hand the area of the unknown subset $\Omega$ times the scaling factor $\lambda$.
An alternative point of view is to  see \eqref{pbbeta} as the minimization  of  the shape functional:
  $J_\lambda:\ \O\subset D \mapsto  J(\O) + (1-\lambda)\, |\O|$ \ where: 
\begin{equation}\label{shape}
 J(\O)=  \inf\left\{ \int_{D\setminus \O} \sqrt{1+|\nabla u|^2}\; dx \ :\ u = 0 \quad\tinm{on} \partial D,\quad u=1 \quad \tinm{on} \partial\O\right\}  .
\end{equation}

\begin{figure}[H]
\centering
\includegraphics[scale=1]{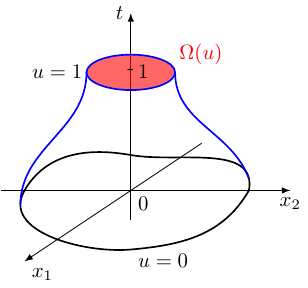}
\caption{A surface with prescribed boundary $u=0$ on $\partial D$ and free boundary $\partial\{u=1\}$.}\label{fminisurface}
\end{figure}
In this context, if $\Omega$ is an optimal set for $J_\lambda$, then a solution $u$ to \eqref{pbbeta} is obtained by solving  \begin{gather}
-\Div \frac{\nabla u}{\sqrt{1+|\nabla u|^2}} = 0 \quad \tinm{in} D\setminus \Omega, \label{curvature}\\ \nonumber
u = 0 \quad\tinm{on} \partial D,\qquad u=1 \quad \tinm{on} \partial\Omega ,
\end{gather}
where \eqref{curvature} encodes the  zero mean curvature of the parametrized surface $z=u(x)$.

In fact, it is well known that the above nonlinear boundary value problem does not always have a solution (see for example \cite{ekeland1974analyse}). In order to start with a well-posed problem, we replace the variational problem \eqref{pbbeta} by the following relaxed formulation, in which the class of competitors $u$ is extended to possibly discontinuous functions (thus violating the boundary conditions):

\medskip
\begin{equation}
\label{betarelax}
\beta(\lambda) = \inf \left\{ E_\lambda(u) \sepa u\in BV(D) \right\}
\end{equation}
where
\begin{align}\label{El}
E_\lambda(u) := \int_D \sqrt{1+|\nabla u|^2}\; dx + \int_D |D^s u| + \int_{\partial D} |u|\; d\HH^{N-1} - \lambda\; |\{ u\ge 1\}|.
\end{align}

Here $BV(D)$ denotes the subspace of functions $u\in L^1(D)$ whose distributional gradient $Du$  is a vector Radon measure of bounded total variation from $D$ to $\RR^N$. The Radon-Nikodym derivative of $Du$ with respect to the Lebesgue measure on $D$ coincides with the a.e. defined approximate gradient $\nabla u$ while $D^s u$ denotes its singular part. Such functions $u$ have a well defined trace  in $L^1(\partial D)$ (see \cite{Suquet}) and a well defined jump set $S_u:=\{u_+>u_-\}$ where $u_-(x),u_+(x)$ are the lower (resp. upper) approximative limits of $u$ at any $x\in D$. It turns out that $S_u$ is an $N-1$ dimensional rectifiable subset of $D$ such that
  $ D^s u = (u_+- u_-) \, \nu_u \HH^{N-1} \res S_u  + D^c u$ where $\nu_u$ is a normal unit vector to $S_u$;   the remainder of $D^s$ on the complement of $S_u$  is the so called Cantor part that we will denote by $D^c u$.  
For further details on the space $BV(D)$, we refer to  \cite{mua1992bv}  and also to \cite{attouch}  for classical relaxation issues in this space. It is easy to check that \eqref{betarelax} admits solutions in $BV(D;[0,1])$ and that the infima defined in
\eqref{pbbeta} and \eqref{betarelax} coincide. However, the solution is not unique (see \cite{attouch-champion} in a simpler context), and characterizing a global minimum turns out to be a very difficult problem. This is due to the nonconvexity of the term $-\lambda|\aco{u \ge 1}|$. In fact, if we restrict the definition of $E_\lambda(u)$ to $u\ge 0$, we can rewrite this term as the integral $\lambda \, \int_D g (u)dx$ where $g: \R\to \R$:
\begin{align}\label{def:g}
g(t)=\begin{cases} -  t & \tinm{if} t\le 0 \\ 
0 &\tinm{if} 0\le t<1,\\
- 1 &\tinm{if} t\ge 1.
\end{cases}
\end{align}
This integrand $g$ is lower semicontinuous non-increasing and admits a jump at $t=1$. Its convexification $g^{**}$
given below is such that $\{g^{**}<g\}=(0,1)$:
\begin{align}\label{co-g}
 g^{**}(t) =  -  \min \{t, 1\} .
\end{align}  
In order to overcome this lack of convexity,  we will use a duality recipe in dimension $N+1$ which was developed  in \cite{boufra2018} for functionals of the kind $u\in W^{1,p}(D) \to \int_D (f(\nabla u)+ g(u))\, dx$  
for general non-convex lower semicontinuous functions  $g:\R\to (-\infty,+\infty]$, possibly admitting isolated discontinuity points. 
In \cite{boufra2018}, the integrand $f$ was assumed to be convex  and satisfying a $p$-growth lower bound with $p>1$. Recently this duality result has been extended to the case $p=1$ where $f$ has linear growth at infinity (see \cite{bouphan2025}). Moreover,  if $f$ is assumed to be positively one homogeneous (for example $f(\nabla u)=|\nabla u|$),
 an  {\em exclusion principle} was derived in \cite{bouphan2025},  namely:  the minimizers of $\int_D (f(\nabla u)+ g(u))\, dx$ under a Dirichlet condition  $u_0\in L^1(\partial D)$ such that $g(u_0)= g^{**}(u_0)$ cannot take values in the set $\{g^{**}<g\}$.
\medskip

We are also interested  in  a variant of \eqref{pbbeta} where
 the surface-area term is replaced by the total variation of $u$. In virtue of the inequalities $|z| \le \sqrt{1+|z|^2} \le 1+|z|$ holding for all $z\in\RR^N$, we  infer that:
 \begin{align}\label{ineqbeta}
m(\lambda,D) := \beta_0(\lambda) -|D|\ \le\ \beta(\lambda)\ \le\ \beta_0(\lambda)\quad ,
\end{align}
where
  \begin{align}
\beta_0(\lambda):=\inf \left\{ \int_D (1+|\nabla u|)dx -  \lambda|\aco{u= 1}| \; : \; u\in W^{1,1}_0(D), \; 0\le u \le 1 \right\}. \label{beta0}
\end{align}
We are led to another free boundary problem where the PDE counterpart of \eqref{curvature} is given by
\begin{gather*}
-\Div \pare{ \frac{\nabla u}{|\nabla u|}  }= 0 \quad \tinm{in} D\setminus \Omega\ ,\end{gather*}
while the relaxed functional to be minimized in $BV(D;[0,1])$ becomes
\begin{align}
E_\lambda^0(u) := |D|+ \int_D |D u| +  \int_{\partial D} |u|d\HH^{N-1} + \lambda\, \int_D g(u)\, dx\ .\label{E0l}
\end{align}
Here, two important observations underscore the role of this variant of the original problem: 
\begin{itemize}[leftmargin=6mm]
\item  The equality $E_\lambda=E_\lambda^0$ holds on the subset $BV(D; \{0,1\})$. Indeed, if $u$ is the characteristic function $\one_\Omega$ of a subset $\Omega\subset D$, then $\Omega$ has a finite perimeter in $R^d$ given by  $P(\Omega)=  \int_D |D u| +  \int_{\partial D} |u|d\HH^{N-1}$ (in that case $\nabla u=0$ a.e and $D^c u=0$). It follows that: 
\begin{equation}\label{E=E_0}
 u=\one_\Omega \ \implies\ E_\lambda(u)=E_\lambda^0(u)= P(\O)-\lambda |\Omega|.
\end{equation}

\item The non convex variational problem $\beta_0(\lambda)$ in \eqref{beta0} satisfies the  one--homogeneity assumption mentioned above. Thus, in virtue of \cite[Theorem 2.2]{bouphan2025} and since $\{g^{**}<g\}=(0,1)$, any solution to $\beta_0(\lambda)$ is of the form $u=\one_\Omega$. A consequence is that  the left hand side of \eqref{ineqbeta} coincides with the minimum of the following shape optimization problem:
\begin{align}
m(\lambda,D)= \min\aco{P(\Omega) -\lambda |\Omega| \sepa \Omega\subset D}. \label{mlD}
\end{align}
\end{itemize}
Note that the equality $\beta_0(\l)= |D| + m(\l,D)$ can be also recovered by using the co-area formula. Actually  the problem \eqref{mlD}  has been studied in \cite{ alter2009uniqueness,alter2005characterization} as being in  close relation with the celebratred Cheeger problem (see the survey \cite{ parini2011introduction}): 
\begin{equation}\label{def:h_D}
 h_D := \inf\aco{\frac{P(\Omega)}{|\Omega|} \sepa \Omega\subset D} .
\end{equation}

In particular, one has  $m(\lambda, D) =0$ for every $\l \le h_D$ while $m(\l,D)<0$ otherwise. 

\bigskip
The  main contributions of this paper are the following:

\begin{enumerate}[leftmargin=4mm]
\item [-] \ We show that the strict inequality $\beta(\lambda) <\beta_0(\lambda)$ holds if
 $u\equiv 0$ or $u\equiv 1$ are not  minimal for \eqref{pbbeta} (Theorem \ref{beta<beta0}). This latter condition is equivalent to require that $\lambda$ is between  two thresholds $\l_0$ and $\l_1$.
 By Lemma \ref{intervals},  these thresholds satisfy the inequalities:
  \begin{align*}
0< \lambda_0\ \le h_D\ \le \frac{P(D)}{|D|}\ \le\ \lambda_1 \ \le \ +\infty .
\end{align*}
More precise estimates are obtained in Proposition \ref{estim2} by using special calibration fields for the dual problem in dimension $N+1$ associated with
\eqref{pbbeta}. Numerical simulations are given  for radial examples in dimension two. 

\item[-] \ We develop a theory of $\theta$-calibrability for solving $m(\lambda, D)$. In the two-dimensional case,  we associate to any bounded  convex  open set $D$  a specific potential  $\rho:\ov D\to (0, R_D]$  where $R_D$ denotes the inradius of $D$. This potential is  continuous, locally Lipschitz in $D$ with  a maximal plateau $\{\rho=R_D\}$ which coincides with the \emph{central subset}  defined by $U_D:= \{x\in D : d(x,D^c)=R_D\}$.  
In  $D\setminus \ov U_D$, $\rho$ satisfies  $\nabla \rho\not=0$ a.e. and solves the  boundary problem:
\begin{align}\label{masterpde} \div \frac{\nabla \rho}{|\nabla \rho|} +  \frac1{\rho} =0 \quad, \quad \frac{\nabla \rho}{|\nabla \rho|}\cdot \nu_D=- 1 \quad \text{on $\partial D$}\, . \end{align}
 A very simple geometric construction of this so-called cut locus potential is described in 
Section \ref{scutlocus}, which allows us to make explicit a calibrating vector field for any $\l\ge h_D$. Then we deduce that the unique solution  of $m(\lambda,D)$  is given by $\Omega_\l = \{\rho> \frac1{\l}\}$. It  coincides with the union of all balls of radius $\frac1{\l}$ contained in $D$ (see \eqref{def:Olambda}),
so that we recover by a different  method the result in \cite{alter2009uniqueness,alter2005characterization}. 
\end{enumerate}

\medskip

Before we finish this introduction, we want to mention   that the  idea of building a potential $\rho$ associated with a general convex set $D$ can  be  extended in higher dimensions. The same PDE as in \eqref{masterpde} would be kept, and solutions to \eqref{mlD} would still be the upper level sets $\{\rho> \frac1{\l}\}$ for every $\l<h_D$. In turn the geometric characterization that we found for $N=2$ (by means of the normal distance to the cut locus of $D$) does not work any more for $N>2$. This suggests a very interesting open issue worth to explore in future work.
 Note that here the potential $\rho$  depends only of the shape of $D$ (see Remark \ref{newpb}), in contrast with the parametrized potentials introduced in \cite{alter2005characterization}.

\medskip
The paper is organized as follows:

\begin{itemize}[leftmargin=6mm]
\item [-]\ In Section  \ref{secdual}, we present briefly the duality recipe for  the non-convex problem $\beta(\l)$. This leads to an $N+1$ dimentional dual  problem  and to optimality conditions in terms of divergence free calibration vector fields.
In the case of the homogenous variant $\beta_0(\l)$,  this dual problem  reduces to a more classical formulation in dimension $N$.
The optimality of a subset of $\O\subset D$ for \eqref{mlD} is characterized by a $\theta$- calibrability condition.
 
 \item[-] \ In Section \ref{secbb0}, we discuss  the occurrence of the inequality $\beta(\l) < \beta_0(\l)$ according to the value of $\l$ and to the geometric properties of $D$.  Some numerical simulations are given in the case where $D$ is of a disk in $\R^2$.
%

\item[-] \ In Section \ref{scutlocus}, we restrict ourself to the case of a convex subset $D\subset \R^2$. After a short background on Cheeger sets, we introduce the  called cut-locus potential $\rho$ and focus on the explicit  construction of  a vector field  $q: D \mapsto \R^2$ which calibrates the unique solution $\O_\l$ of  \eqref{mlD}.
 
\end{itemize}

%

\section{Dual problems and calibrations}
\label{secdual}

The variational problems associated respectively with $\beta(\lambda)$ and $\beta_0(\lambda)$ are non convex with linear growth. They enter the duality framework developed initially in $W^{1,p}(D)$ for $p>1$  (see \cite{boufra2018}), further extended to $W^{1,1}(D)$ and $BV(D)$  in \cite{bouphan2025}, where $D$ is a bounded domain in $\R^N$. In these references a characterization of global minimizers is provided by using calibration fields defined in $L^\infty(D\times I;\R^{N+1})$ 
 where $I$ is an open interval of $\R$ such that $\ov I$  contains the range of all solutions.  
 This section is devoted to state the dual problems in the particular case of $\beta(\l)$ and $\beta_0(\l)$. 
 In both cases,  the solutions range in $[0,1]$ so that we take $I=(0,1)$ and the searched calibration fields are  of the form
 $$ \sigma(x,t) =\left(\sigma^x(x,t), \sigma^t(x,t)\right) \in \R^N\times \R  \quad \text{where $(x,t)\in D\times (0,1)$} . $$
In the following we will denote  
$$   Q := D \times (0,1) .$$

\subsection{Dual problem of \eqref{pbbeta}}\label{dualb0}


Following \cite{boufra2018}, the dual formulation of  the variational problem of \eqref{pbbeta} (or of its relax form \eqref{betarelax}) is given by
the following formulation in $Q$ (thus in dimension $N+1$).
\begin{align}\label{dual-beta}
\sup \aco{ -\int_D \sigma^t(x,0)\, dx \sepa \Div\sigma =0 \tinm{in} Q,\; \sigma^t + \sqrt{1- |\sigma^x|^2} \ge 0 \tinm{in} Q,\; \sigma^t(x,1) \ge \lambda -1 \tinm{on} D},
\end{align}

Adopting a fluid dynamics view point, 
 \eqref{dual-beta} can be interpreted as the maximization of the  downward flow $\sigma=(\sigma^x,\sigma^t)$ of an incompressible fluid ($\Div\sigma =0$ in $Q$) through the bottom interface $D\times\{0\}$ when it is subject to the pointwise non linear constraint $\sigma^t + \sqrt{1- |\sigma^x|^2} \ge 0$ a.e. in $Q$ while $\sigma^t(x,1) \ge \lambda -1$ on the upper interface $D\times \{1\}$.
 
\begin{theorem}\label{duality1} The duality principle given  in \cite{bouphan2025} leads to the following no-gap equality
$$   \beta(\l) \ =\  \sup \eqref{dual-beta} .$$
\end{theorem}

\begin{remark}
In order to explain how the results in \cite[Secton 3.3]{boufra2018} are applied to our case for deriving the dual problem \eqref{dual-beta}, two comments are in order:

\begin{enumerate}
\item [-]  Following the general notations of \cite{boufra2018}, the original primal problem \eqref{pbbeta}, can be written as
$$ \inf \left\{ f_\l(u, \nabla u) \, dx\ :\ u\in W^{1,1}_0(D) \right\} ,  $$
where $f_\l(t,z):= \sqrt{1+ |z|^2} + \l g(t)$ and $g$ is the non convex function in \eqref{def:g}.  
Then, the  bulk constraint  on $\sigma$ appearing in the dual problem (the condition (3.20) in \cite{boufra2018}  can be written as $f_\l^*(t, \sigma^x)\le \sigma^t$ where $f_\l^*(t,\cdot)$ is the Fenchel conjugate of $f_\l(t,\cdot)$. In our case, we have 
$$  f_\l^*(t, \sigma^x)= \l\, g(t)  - \sqrt{1- |\sigma^x|^2} \quad \text{if $|\sigma^x|\le 1$}\ , +\infty \quad \text{otherwise}.$$
where the non convex function $g$ vanishes on $[0,1)$. Hence the bulk condition holding a.e. in $Q$ reduces to $\sigma^t + \sqrt{1- |\sigma^x|^2} \ge 0$. 

\item[-] In view of the discontinuity of $g$ at $t=1$, an additional normal trace condition has to be imposed on the boundary interface $t=1$
namely $\sigma^t(x,1) \ge \inf f_\l^*(1, \cdot) =- f_\l(1,0)$ (see condition (3.21) in \cite{boufra2018}). In our case, we obtain 
the $\l$-dependent constraint $ \sigma^t(x,1) \ge \l-1$, which accounts the free boundary associated with the subset $\{u=1\} .$ 
Note that this condition holding a.e. $x\in D$ is well defined  in the sense of normal traces of bounded  functions $Q \to \R^d$ whose  distributional divergence in $\DD'(Q)$ belongs to $L^\infty(Q)$ (see \cite{Abis, Anzelloti1984}).
\end{enumerate}
\end{remark}

\subsection*{Optimality conditions for \texorpdfstring{$\beta(\lambda)$}{beta}}

Let $u \in BV(D; [0,1])$ and $\sigma$  be an admissible vector field $\sigma\in L^\infty(Q;\RR^N \times \RR)$ for  \eqref{dual-beta}. Then, by the no-gap identity $\beta(\lambda)= \sup \eqref{dual-beta}$ (see \cite[Thm 3]{boufra2018}) and recalling \eqref{El},  the optimality of an admissible pair $(u,\sigma)$ is equivalent to the equality 
 $$E_\l(u)=  -\int_D \sigma^t(x,0) \, dx  .$$
Accordingly, by localizing this relation, we obtain the following set of optimality conditions:

\begin{equation}\label{admi-beta}
\Div\sigma =0\, \tinm{in} \, \DD'(Q)\, ,\ \sigma^t + \sqrt{1- |\sigma^x|^2}\ge 0\ \text{in $Q$ a.e.}\ ,\ \sigma^t(x,1) \ge \lambda -1 \, \tinm{a.e. in $D$} ,
 \end{equation}
 
 and the pointwise conditions  on the completed graph of $u$:
%
%
\begin{equation}\label{opti-beta}
 \left\{
 \begin{array}{llll}
\sigma^x(x,u(x)) =&\; \frac{\nabla u(x)}{\sqrt{1+|\nabla u(x)|^2}} && \tinm{ $\LL^N$-a.e. on}\ \{u <1\} ;
\\
\sigma^t(x,u(x)) =&\; - \sqrt{1+|\nabla u(x)|^2} && \tinm{ $\LL^N$-a.e. on }\ \{u <1\} ;
\\
\sigma^t(x,1) =&\; \lambda -1 && \tinm{ $\LL^N$-a.e. on }\ \{u=1\}; 
\\
\sigma^x(x,t)\cdot\nu_u =&\; 1  && \tinm{ $\HH^{N-1}$-a.e. on $S_u$ \ ,\ $\forall t \in [u_-(x), u_+(x)]$};
\\
\sigma^x(x,\tilde u(x)) =& \; 1 && \tinm{ $|D^cu|$-a.e.   }
\end{array}
\right.
\end{equation}
where:
\begin{enumerate}
\item [-] $\nabla u(x)$ is the  a.e defined approximate gradient of $u$; 
\item [-] $u_-(x), u_+(x)$ denote the lower and upper approximative limits of $u$ ; if $u_-(x)= u_+(x)$ we denote by $\tilde u(x)$ the common value;
\item[-] $S_u = \{u_+>u_-\} $ stands for the \emph{$N-1$-rectifiable} jump set of $u$; 
\item[-] $\nu_u$ the oriented unit normal vector to $S_u$ ;
\item [-] $D^c u$ is the Cantor part of the vector measure $Du$ (which has no mass on $S_u$). 
\end{enumerate}

\begin{remark}\label{background} Here a short  background is in order. 
We recall that $BV(D)$ is the set of functions $u$ in $L^1(D)$ whose distributional gradient $Du$ is an element of $\mathcal {M}(D;\R^N)$ the set of  Radon vector measures from $D$ to $\R^N$. By considering only elements  $u\in BV(D)$ ranging into the closed interval $[0,1]$,
 we obtain a closed subspace denoted $BV(D;[,1])$.
For every function $u\in BV(D)$,  $Du$ is a bounded Radon measure which can be decomposed into 
\begin{align}\label{decompo}
Du = \nabla u dx + D^c u + (u_+ - u_-) \nu_u d(\HH^{N-1}\res S_u)
\end{align}
where $\nu_u$ denotes the Radon-Nikod\'ym density of   $Du$ with respect to its total variation $|Du|$, i.e.~$\nu_u:= dDu/d|Du|$. Note that $D^cu$ is the Cantor part of the measure $Du$. The quantity $[u]:= u^+ - u^-$ is called the {\it jump} of $u$ across the interface $S_u$ and the direction of the jump is given by $\nu_u$ along $S_u$. Accordingly, the {\it complete graph} of function $u$, denoted by $\ov G_u$, is  defined by 
\begin{align*}
\ov G_u:= \bigcup_{x\in\Omega} \Big( \{x\}\times[u_-(x), u_+(x)] \Big).
\end{align*}
For futher details , we refer to the monograph \cite{mua1992bv}.
\end{remark}

It is an $N$-rectifiable subset of $\Omega\times\RR$ with an oriented unit normal denoted by $\widehat \nu_u$. This oriented  normal $\widehat\nu_u$ is $\HH^N\res\ov G_u$ a.e. determined by 
\begin{align}\label{hatnus}
\widehat\nu_u (x,t) = (\nu_u(x),0) \quad \tinm{for} x\in S_u \tinm{and} t\in [u_-(x), u_+(x)],
\end{align}
on the vertical part of $\ov G_u$, whereas, on the {\it approximately continuous part}  $G_u:=\{ (x, \tilde u(x))\}$, it is identified as
\begin{align}
\widehat\nu_u (x,\tilde u(x)) = \frac{(\nabla u(x), -1)}{\sqrt{1+ |\nabla u(x)|^2}}\label{hatnur}
\end{align}
if $u$ is approximately differentiable at $x$ (with its approximate gradient $\nabla u(x)$), and it is horizontal, i.e. $\widehat\nu_u (x,\tilde u(x))= (\nu_u(x),0)$, at points in the support of the Cantor part of $Du$. Notice that $\nu_u= dD^c u/ d|D^c u|$ holds $|D^c u|$-a.e in $D$. We remark also that the complete graph $\ov G_u$ of functions $u$ belonging to $W^{1,1}(D)$ agrees with the {\it continuous graph} $G_u$ , on which $\widehat \nu_u(x,u(x))$ is given by \eqref{hatnur}.
Therefore, the normal trace of the calibration field $\sigma$ given in \eqref{opti-beta} satisfies the equality  $\sigma \cdot \widehat \nu_u =1$
 on $G_u$ ($\HH^N$-a.e.).

\subsection{Dual problem of \texorpdfstring{$\beta_0(\lambda)$}{beta0}}

By applying the same duality framework as for $\beta(\l)$, we obtain the relation:
\begin{align*}
\beta_0(\lambda) = \sup \aco{\int_D-\sigma^t(x,0)dx \sepa \Div\sigma =0,\; |\sigma^x| \le 1,\; \sigma^t +1  \ge 0 \tinm{in} Q,\; \sigma^t(x,1) \ge \lambda -1 \tinm{on} D} .
\end{align*}
We recall that $m(\l,D)= \beta_0(\l) - |D| ,$ where $m(\l,D)$ is the minimum of the  shape optimization problem \eqref{mlD}.
In order to simplify further computations, it is convenient to rewrite the duality relation above, after changing $\sigma^t$ into $\sigma^t +1$, as follows:
\begin{align}\label{dual-beta0}
m(\l, D) = \sup \aco{\int_D-\sigma^t(x,0)dx \sepa \Div\sigma =0,\; |\sigma^x| \le 1,\; \sigma^t  \ge 0 \tinm{in} Q,\; \sigma^t(x,1) \ge \lambda  \tinm{on} D} .
\end{align}
Since the pointwise constraints over $\sigma^t$ and $\sigma^x$ are decoupled, we can easily construct particular admissible calibrations $\sigma$. An important subclass is associated to the subset 
$$S_\l :=  \{ q\in L^\infty(D;\R^N)\ :\ |q| \le 1 ,\  0\le \Div q \le \l \}.$$
It is easy o check that $S_\l$ is convex and wealy* closed in $L^\infty(D;\R^d)$. Moreover, to each element $q\in S_\l$, we associate a competitor for \eqref{dual-beta0}, namely:
$$ \sigma_q (x,t) :=\  \left(-q(x), \l +  (t-1)\, \Div q(x)\right) .$$
It follows that:
$$ m(\l, D)  \ \ge\ \sup \left\{ \int_D ( \Div q- \l )\, dx \ :\  q \in S_\l \right\}.$$  
In turn, we are going to show that the latter inequality is an equality.

\subsection*{Reduction to dimension $N$ and optimality conditions for $\beta_0(\l)$ }

By the exclusion principle mentioned in the introduction, optimal solutions $u$ for \eqref{beta0} are of the form
 $\one_\Omega$ for some set  $\Omega\subset D$. This set $\Omega$ has finite perimeter and solves the geometrical problem $m(\l, D)$.

\begin{theorem}\label{duality2}
Let $D$ be a bounded domain of $\R^N$ with Lipschitz boundary. Then 

\begin{itemize}
\item [(i)]\  we have the following duality relation
\begin{align}
 m(\lambda,D)=\sup \aco{\int_D (\Div q -\lambda) dx \sepa q\in L^\infty(D;\RR^N),\; |q|\le 1,\; 0 \le \Div q \le \lambda } ,  \label{dual2}
\end{align}
where  the supremum in the right hand side is attained. 

\item [(ii)]\ a pair $(\O,\ov q)$ solve \eqref{mlD} and \eqref{dual2} respectively if and only if the following conditions are satisfied
\begin{align}
& |\ov q| \le 1 \quad\tinm{a.e. in} D, && 0 \le \Div \ov  q \le \lambda \quad\tinm{a.e. in} D,\label{opt1}\\
&  \ov q\cdot \nu_{\Omega} =1 \quad\tinm{$\HH^{N-1}$-a.e. on} \partial\Omega,  && \Div \ov  q =\lambda \quad\tinm{a.e. in} D\setminus \Omega.\label{opt2}
\end{align}
\end{itemize}
\end{theorem}

\begin{remark}\label{trace} \ In the left hand side of \eqref{opt2}, $\nu_\O$ stands for the exterior normal vector if $\O$ has a smooth boundary;
if $\O$ is merely a subset with finite perimeter, the reader should agree that the equality means that the measure associated with the duality bracket
$< \ov q,  D\one_\O>$ in the sense of G.Anzelloti \cite{Anzelloti1984} coincides with  $\HH^{N-1}\res \partial_* \Omega$ being $\partial_* \Omega$ the reduced boundary of $\Omega$. Note that $\nu_\O$ can be defined  $\HH^{N-1}$ on  $ \partial \O$ (see \cite{EG}).

\end{remark}

\begin{proof} First we notice that, with  $g^{**}$ defined by \eqref{co-g}, we have the equality
$$  m(\l,D) = \inf \left\{ \int_D (|\nabla v| + \lambda \, g^{**}(v))\, dx\ :\ v\in W^{1,1}_0(D)\right\}.$$
Indeed, since $g^{**}(t) \ge g^{**}( \min\{t_+,1\})$, by truncating the competitors, the infimum above is unchanged if we restrict to $v$ such that $0\le v\le 1$. For such a $v$, we have $g^{**}(v)=-\l\, v$ so that, by the coarea formula:
$$\int_D (|\nabla v| + \lambda \, g^{**}(v))\, dx = \int_0^1 \left(P(v>t) - \l |v>t|\right)\, dt \ge  m(\l, D) .$$
Next we apply classical convex duality arguments. The perturbation function
$$ h: p\in L^1(D,\R^N) \mapsto   \inf \left\{ \int_D (|p+ \nabla v| + \lambda \, g^{**}(v))\, dx\ :\ v\in W^{1,1}_0(D)\right\} \, ,$$
is convex and continuous. Indeed $h(p) \le m(\l,D) + \int_D |p|\, dx$. Therefore it holds
$$   h(0) = - \min \{h^*(q) \ :\ q\in L^\infty(D;\R^N)\} ,$$
being $h^*$ the Fenchel conjugate of $h$ in the duality between $L^1(D,\R^N)$ and $L^\infty(D,\R^N)$. Let us compute
\begin{align*} h^*(q)  &=  \sup \aco{\int_D p\cdot q \, dx - h(p) : p\in L^1(D;\R^N)} \\
&=  \sup \aco{\int_D \left(p\cdot q  - |p+\nabla v| - \l g^{**}(v) \right)\, dx : (v, p)\in W^{1,1}_0(D)\times L^1(D;\R^N)}\\
 &=\sup \aco{\int_D \left(\tilde p\cdot q  - |\tilde p|\right)  +   \int_D \left(v \,  \Div q - \l g^{**}(v)\right)  \, dx  : (v,\tilde p)\in W^{1,1}_0(D)\times L^1(D;\R^N)} \\
 &= \chi_{|q|\le 1} + \int_D  \l\,  g^*\left(\frac{\Div q}{\l}\right)  \, dx .
 \end{align*}
 Here above:
 \begin{enumerate}
\item [-] to pass from the second line to the third, we wrote $p\cdot q  - |p+\nabla v|$ as $ \tilde p\cdot q  - |\tilde p| - \nabla v \cdot q$ where 
$\tilde p = p+\nabla v$ and then  we used the integration by parts  $-\int \nabla v \cdot q= \int v \,  \Div q$ ;  
\item [-]to pass to the last line, we decoupled the supremum in $\tilde p$ from that  in $v$ and compute them as Fenchel conjugates of integral functionals, 
taking into account that the Fenchel conjugate of the norm is the indicator funtion of the unit ball while the conjugate of  $\l g^{**}$ is
 $\l\, g^*(\frac{\cdot}{\l})$.  
\end{enumerate}
A straightforward computation shows that 
$$\l\, g^*\left(\frac{t^*}{\l}\right)= \begin{cases}  t^* +\l  & \text{ if $-t^* \in [0,\l]$} \\ +\infty & \text{otherwise} \end{cases} $$
It is convenient now to change $q$ into $-q$. We arrive to the simple expression:
$$ h^*(-q) =  \int_D (\l -\Div q)\ dx \quad \text{if $0\le \Div q\le \l$ \ a.e.}  \quad,\quad  h^*(-q) =+\infty \quad \text{otherwise}.$$
Since $h(0)= m(\l,D)$, the equality $h(0)= - \min h^*$ leads to the equality stated in \eqref{dual2} where the supremum is actually a maximum. 
The assertion (i) is proved.

\medskip
Let us now establish the assertion (ii). Let  $(\one_\O, q)$ be an admissible pair. This means that $\O$ has finite perimeter and that $q$ satisfies \eqref{opt1}.
In particular, the condition $|\ov q| \le 1$ implies that 
$$ P(\O) - \int_\O \Div \ov q\, dx=  \int_{\partial \O} (1- \ov q\cdot \nu_\O)\, d\HH^{N-1}  \ge 0 \ ,$$
while the equality holds if and only if  $\ov q\cdot \nu_\O =1$ holds $\HH^{N-1}$-a.e in $\partial \O$. 

Now, in virtue of \eqref{dual2}, the optimality of $(\one_\O, \ov q)$ is equivalent to the extremality relation:  
$$   P(\O) - \l \, |\O| \ =\  \int_D (\Div \ov q - \l)\, dx ,$$
which we can rewrite as the equality
$$ \left (P(\O) - \int_\O \Div \ov q \, dx\right) \ + \ \left(\int_{D\setminus\O} (\l - \Div \ov q)\, dx\right) \ =\ 0  .  $$
Since above we get the sum of two non negative terms,  
the equivalence with \eqref{opt2} follows. 

\end{proof}

In Section \ref{scutlocus},  an explicit construction of an optimal field $\ov q$ will be given for $D$ being any convex body in dimension 2 and for any $\l <h_D$.

\subsection{Relations with Cheeger problem and calibrability notions}\label{calibrable}

We recall that the Cheeger constant of a bounded domain $D\subset \R^N$ is given by
\begin{align}
h_D:=\inf \left\{  \frac{P(\Omega)}{|\Omega|}\ :\ \Omega\subset D, |\Omega| > 0 \right\} . \label{CheegerConst}
\end{align}
It is well known that the infimum here is actually a minimum and 
any optimal $\Omega$ for  this geometric problem is called a Cheeger set of $D$. In case $D$ is convex, it turns out that this set is unique (see for instance \cite{caselles2007uniqueness}, \cite{leonardi2015overview}, \cite{parini2011introduction}) and we shall denote it by $C_D$. For a non convex $D$, the uniqueness  fails; however, since a union of Cheeger sets is still a Cheeger set, the definition of $C_D$ can be extended 
 (see \cite{leonardi2015overview}, \cite{buttazzo2007selection})  by setting 
\begin{align*}
C_D:=\bigcup \;\{ \Omega \sepa \Omega \tinm{is a Cheeger set of} D\}.
\end{align*}
In this case, $C_D$ is called the \emph{ maximal Cheeger set}.
If $D$ has finite perimeter, we will use in many places an upperbound for $h_D$ , namely the ratio 
\begin{equation} \label{def:lD}
\l_D:= \frac{P(D)}{|D|}\ .
\end{equation}
 It satifies the  inequality $ \l_D \ge h_D$ while the equality holds
 if and only if $D$ is a Cheeger set of him self;  for brevity,  we will say that $D$ is \emph{self-Cheeger}. 
 
 \medskip
 Lets us now  come back to the parametrized geometric optimization problem \eqref{mlD}. It consists, for every $\l\ge 0$, in minimizing the shape functional
 $$ J_\l (\O)\ :=\ P(\O) - \l |\O| \ ,\  \text{ $\O$  measurable subset of $D$}. $$
 Here we set  $P(\O)=+\infty$ if $\O$ is not of finite perimeter and by convention $P(\O)=0$ if $|\O|=0$. In particular, we have 
 $  m(\l,D) = \inf J_\l \le 0$  for every $\l\ge 0$. 
 
 The existence of  minimizers for $J_\l$ follows directly from the direct method of Calculus of Variations. As it appears in the proof of Theorem \ref{duality2}, 
 $m(\l,D)$ is  also the mimimum of the following companion convex problem:
 $$\widetilde m(\l,D) := \min \left\{ F_\l(u):= \int_{\R^N} |Du| - \l \int_D u\ dx  \ :\ u\in BV(\R^N;[0,1]), u=0\ \text{a.e. on\ } D^c\right\} \ .$$
 The equality  $\widetilde m(\l,D)= m(\l,D)$ is a consequence of the corea formula  $ F_\l(u) =\int_0^1 J_\l(\{u>t\})\,dt \ge m(\l,D) $ from which it follows that:
   $$   u \ \text{solves}\  \widetilde m(\l,D) \iff  \{u>t\}\ \text{solves}\  m(\l,D) \ \text{for a.e. $t\in (0,1)$}.$$
  
  The main  properties of the function $m(\l,D)$ are  summarized below:

\begin{proposition} \label{monotony}  Let $D\subset \R^N$ be a bounded set with finite perimeter. Then
 \begin{itemize} 
\item [(i)]\ The function $m(\cdot,D)$ is concave  continuous non increasing on $\R_+$ and satisfies
\begin{equation}\label{ineqmlD}
- (\l- h_D)_+ \, |D|  \ \le m(\l,D) \ \le \min\{0,P(D)-\l |D|\} .
\end{equation}
   Therefore  it holds   $m(\l,D) =0$ for every  $\l\in [0, h_D]$ whereas:
\begin{equation}\label{mlD=0} 
  m(\l,D) <0 \  \text{ for $\l >h_D$} \quad \text{and} \quad \lim_{\l\to +\infty} m(\l,D) + \l |D|=P(D) .
\end{equation}

\item [(ii)]\ Assume that $D$ is a minimizer of $m(\l^*,D)$ for a suitable value $\l^*>0$. Then $\l^*\ge h_D$ and $D$ is the unique minimizer of $m(\l,D)$ for any $\l>\l^*$.
In particular the concave function $m(\cdot,D)$ has the linear behavior $m(\l,D)= m(\l^*,D) - (\l-\l^*)\, |D|$ for every  $\l\ge \l^*$.

\med
\item [(iii)]\  Let $\O\subset D$. Then we have $ m(\l,D)\le m(\l,\O)$ while 
$ m(\l,D)= m(\l,\O)$ if $\O$ is a minimizer of $m(\l,D)$. 
If such is the case with a set $\O$ of positive measure, then the ratio  $\l_\O= \frac{P(\O)}{|\O|}$ satisfies  $\lambda_\Omega < \lambda$.

\end{itemize}
\end{proposition}

\begin{proof} In the assertion (i), the fact that $m(\cdot,D)$ is monotone is obvious, while the concavity follows from writing $m(\cdot,D)$ as the infimum of the family of affine functions $\l\to P(\O)-\l |\O|$,  when $\O$ runs over subsets $\O\subset D$. The relations \eqref{ineqmlD} are obtained by noticing that $P(\O) \ge h_\O |\O|$. Thus,
since $|\O|\le |D|$, the shape functional $J_\l$ to be minimized satisfies
\begin{equation}\label{ineqJ*}
- (\l- h_D)_+ \, |D| \le \, (h_D-\l) \, |\O|\ \le  J_\l(\O)\ \, \le m(\l,D) .
\end{equation}
  Then, by taking the infimum in $\O$ (and  recalling that $m(\l,D) \le 0$), we get \eqref{ineqmlD} which proves that $m(\cdot,D)$ is finite, hence continuous by the concavity property.
If $\l\le h_D$, the first inequality in \eqref{ineqmlD} implies that $m(\l,D) \ge 0$, hence   $m(\l,D)= 0$ since we know that  $m(\cdot,D)\le 0$. 
If $\l> h_D$, we obtain directly that $m(\l,D)<0$ since $J_\l(\O)= |\O| (h_D- \l) <0$ holds whenever $\O$ is a  Cheeger set of $D$. 
Moreover, the second inequality in \eqref{ineqmlD} implies that
 \begin{equation}\label{upper}
\limsup_{\l\to \infty}\,  m(\l,D) + \l |D|   \le P(D).
\end{equation}
In the opposite direction, let $(\l_n)$ be any sequence such that $\l_n\to +\infty$ and choose a minimizer $\O_n$ for $m(\l_n,D)$.
Then, by the upperbound inequality \eqref{upper}, we have
$$ \limsup_{n\to\infty} \left(P(\O_n)  + \l_n |D\setminus \O_n|\right) = \limsup_{n\to\infty} \left(m(\l_n,D)  + \l_n |D|\right) \le P(D) <+\infty . $$
It follows that $\one_{\O_n} \to 1$ in $L^1(D)$ while $\limsup_{n\to\infty} P(\O_n)\le P(D)$.  By the lower semicontinuity of the perimeter, we infer that
 $ P(\O_n) \to P(D)$  so that  $ \l_n |D\setminus \O_n| \to 0$. Finally, we have proved that
 $m(\l_n,D)  + \l_n |D|  \to P(D)$ as wished.
 The proof of assertion  (i) is finished.

\medskip
Lets us turn  to the assertion (ii).  If  $D$ solves $m(\lambda^*,D)$, then $\l^*\ge h_D$ since otherwise, by the assertion (i), we would have
$m(\l^*,D)= J_{\l^*}(D) =0$ which is uncompatible with the inequality $J_{\l^*}(D) \ge (h_D-\l^*) |D| $.  
Next we observe that, if $\l\ge \l^*$, then for every $\O\subset D$
\begin{align*}
J_\l(\O) \ =\ J_{ \l^*} (\O) + (\l^*-\l)|\O|
\ \ge\  J_{ \l^*} (D) + (\l^*-\l)|D|\  =\ J_\l(D) .
\end{align*}
It follows that $D$ is  optimal also for $m(\l,D)$ and that  $m(\l,D)= m(\l^*,D) - (\l-\l^*)\, |D|$.
Futhermore, in view of the inequalities above,  the optimality of a competitor  $\O$ for $m(\l,D)$  requires that $(\l^*-\l) |\O| = (\l^*-\l) |D|$.
If we assume that $\l>\l^*$, this is possible only if $|\O| = |D|$ which means that $\O=D$. Hence $D$ is the unique solution to $m(\l,D)$.


\medskip
Let us prove now the assertion (iii). If $\O\subset D$, the inequality $ m(\l,D)\le m(\l,\O)$  is clear since any admissible subset for $m(\l,\O)$
is also a competitor for $m(\l,D)$. In particular, if $\O$ itself is optimal for $m(\l,D)$, then we have $m(\l,D)=m(\l,\O)= J_\l(\O)$. 
If moreover $\l> h_D$, then $|\O |>0$ and $J_\l(\O)=m(\l,D) <0$. Therefore $J_\l(\O)= |\O| (\l_\O -\l) <0$, whence $\lambda_\Omega < \lambda$ as claimed.
\end{proof}

%
%
%

Owing to Proposition \ref{monotony}, we know that, for $\l<h_D$, the unique solution (in the sense a.e.) of \eqref{mlD} is the empty set , while when $\l=h_D$, we need to add any Cheeger subset of $D$.
For  $\l> h_D$, any solution $\O$ has a positive measure.  Then the optimality conditions obtained in the assertion (ii) of Theorem  \ref{duality2} 
can be exploited  to extend to any $\l>\l_D$ the  notion of  calibrability as it was introduced in \cite{alter2005characterization} for the specific case $\l=\l_D$ .

\begin{definition}\label{def:tcalib}
Let $\theta \ge 1$ and $\Omega \subset \RR^N$ be a bounded set of finite perimeter. $\Omega$ is called $\theta$-{\it calibrable} if there exists a vector field $q\in L^\infty(\Omega;\RR^N)$ such that 
\begin{align*}
\lVert q \rVert_\infty \le 1,\qquad q\cdot\nu_\Omega = 1 \;\tinm{$\HH^{N-1}$-a.e. on}\; \partial\Omega, \qquad 0\le \Div q \le \theta\lambda_\Omega  \;\tinm{in}\; \DD'(\Omega).
\end{align*} 
\end{definition}

\index{t@$\theta$-calibrability}
Obviously, if $\Omega$ is $\theta$-calibrable then it is also $\gamma$-calibrable for every $\gamma\ge \theta$. Accordingly, we define the {\it calibration constant} of a subset $\Omega\subset\RR^N$ as  \index{calibration constant}
\begin{align}
\theta_\Omega := \inf\aco{ \theta \sepa \theta \ge 1,\; \Omega \tinm{is} \theta\text{-calibrable} } , \label{theta_O}
\end{align} 
with the convention that $\theta_\O =+\infty$ if the subset above is void.

\begin{remark}\label{theta-min-1}  If $\theta_\O<+\infty$, the infimum in \eqref{theta_O} is actually a minimum. Indeed, given a sequence such that $\theta_n \searrow \theta_\O$, we can associate $q_n$ in the unit ball of $L^\infty(\O;\R^N)$ such that $q_n\cdot\nu_\Omega = 1$ and $ 0\le \Div q_n \le \lambda_n$.
Up to a subsequence, we have $(q_n,\Div q_n) \sto (q,\Div q)$ for a suitable $q$ such that $|q|\le 1$ and  $ 0\le \Div q \le \theta_\O$, while
$q\cdot\nu_\Omega = 1$ on $\partial D$ by the weak* convergence of the normal traces $q_n\cdot\nu_\Omega \sto q_n\cdot\nu_\Omega$ in $L^\infty(\partial\O)$
(see \cite{Anzelloti1984}).

\end{remark}

\begin{remark}\label{1calib}
In Definition \ref{def:tcalib},  the condition $q\cdot\nu_\Omega = 1$ $\HH^{N-1}$-a.e. on $\partial\Omega $ can be understood as $q \cdot D \one_\Omega = -|D \one_\Omega|$ in the sense of measure in $\ov\Omega$. By integrating by parts, we infer that:
\begin{equation}\label{theta>1}
 P(\O)\ =\ \int_{\partial\O} -  q \cdot D \one_\Omega\ =\ \int_\O \Div q\ \le\ \theta \l_\O \, |\Omega| \ =\ \theta P(\O) .  
\end{equation}
As a consequence,  the condition $\theta\ge 1$ is necessary if we wish  the notion above to be non empty.
Moreover, if $\theta=1$, then the equality \eqref{theta>1} implies that $\int ( \l_D - \Div q)=0$, hence $\div q=\l_D$ a.e. in $D$.
In this case we recover, after changing $q$ in $-q$,  something very similar to the definition of calibrability  in \cite[Definition 1]{alter2005characterization}, except that we do not need to impose the condition $\div q = \l_\O \one_\O$ on all $\R^N$. Therefore, our  \emph{$1$-calibrability} condition is weaker than the one
proposed in  \cite{alter2005characterization} (at least when $\O$ is not convex) .

\end{remark}

\begin{remark}\label{calib-Cheeger}  It is important to keep in mind that the $1$-calibrability property characterizes sets $\O$ which are self-Cheeger. In other words:
\begin{equation}\label{1calib=Cheeger}
\theta_\O = 1 \ \iff\ \l_\O=h_\O \ \iff \ \text{$\O$ is self-Cheeger}.
\end{equation}
  Indeed, if $\theta_\O=1$, as noticed in Remark \ref{1calib}, it exists $q\in L^\infty(\O;\R^N)$ such that $|q|\le 1$ and $\Div q= \l_D$.
Then, for every $A\subset \O$ with positive measure, we have:
$$ P(A) \ge \int_A \div q\, =\,  \l_\O \, |A| \ \implies \ \frac{P(A)}{|A|}  \ge \frac{P(\O)}{|\O|} .$$
It follows that $A=\O$ solves the Cheeger problem in $\O$. The converse is trivial since, in that case, $\l_\O=h_\O$ by assumption.  

\end{remark}

Our notion of $\theta$-calibrability can be used first to improve  \cite[Proposition 2]{alter2005characterization},
without any convexity assumption.

\begin{proposition}\label{mlDtheta}
Let $\theta \ge 1$, and $\Omega \subset \RR^N$ be a bounded  set of finite perimeter. The following assertions are equivalent:
\begin{itemize}
\item[(i)] $\Omega$ is $\theta$-calibrable,
\item[(ii)] $\Omega$ is a minimizer of $m(\theta\lambda_\Omega,\Omega)$.
\end{itemize}
\end{proposition}

\begin{proof} In view of Definition \ref{def:tcalib}, it is enough to apply the assertion (ii) of Theorem \ref{duality2} taking $D=\O$ and $\l= \theta\, \l_D$.

\end{proof}

Regarding the original problem \eqref{mlD}, we obtain the following result:
\begin{corollary}\label{theta-min}
Let $\Omega\subset D$ with positive measure. Then $\O$ is $\theta$-calibrable for a suitable constant $\theta\ge 1$ if and only if it solves  $m(\lambda^*,D)$ for some $\lambda^*\ge h_D$. In this case,  the calibrability constant of $\O$ is given by:
\begin{equation}\label{minlambda}
 \theta_\O \ =\ \min \left\{ \frac{\l}{\l_\Omega}\ :  \ \text{$\O$ solves $m(\l,\O)$ } \right\} .
\end{equation}

\end{corollary}

\begin{proof} 
Suppose that $\Omega$ solves problem $m(\lambda^*,D)$. If  $\l^*=h_D$, then the assumption $|\O|>0$ implies that $\O$ is a Cheeger set of $D$. Therefore, by \eqref{1calib=Cheeger}, $\O$ is $1$-calibrable and $\theta_\O=1$. If $\l^*>h_D$, we know that $m(\l^*,D)<0$. Then, by the assertion (iii) of Proposition \ref{monotony}, it follows  that  $m(\l^*,D)= m(\l^*,\O)$ and $\lambda_\Omega < \lambda^*$. Therefore, in virtue of Proposition \ref {mlDtheta}, 
 $\O$ is $\theta^*$-calibrable for the constant $\theta^*= \frac{\lambda^*}{\lambda_\Omega} >1$. 
 
 Conversely, let us assume that $\O\subset D$ is  $\theta$- calibrable for  some $\theta \ge 1$. Then by Proposition \ref {mlDtheta}, $\O$ solves $m(\l,D)$ for $\l=\theta\, \l_D$.  Accordingly we can conclude with the equality characterizing the calibrability constant $\theta_\O$ when it is finite.  
 
\end{proof}

The characterization of $\theta$-calibrable sets among the class of finite perimeter subsets in $\RR^N$ is a difficult issue.
In the case of a convex set $\O\subset \R^N$, we have the following result directly deduced from \cite[Theorem 9]{alter2005characterization}
and Corollary \ref{theta-min}.

\begin{proposition}\label{chathetacalib}
Let $\Omega\subset\RR^N$ be convex, bounded of class $C^{1,1}$. Let  $\kappa_\infty(\partial \Omega)$ denote the $L^\infty$-norm of the mean curvature of $\partial \Omega$. Then the calibrability constant of $\O$ is given by: 
\begin{equation}\label{tcalib=}
\theta_\O = \max \left\{1, \frac{(N-1)}{\l_\O}\, \kappa_\infty(\partial \Omega) \right\}
\end{equation}
\end{proposition}

\section{Comparison results} 
\label{secbb0}

In this section, we focus  on the initial question raised in the introduction about comparing the free boundary problems \eqref{pbbeta}
and \eqref{beta0}, for a given value of the parameter $\l>0$. Recall that  the respective infima of these problems
  $\beta(\lambda)$ and $\beta_0(\lambda)$ always satisfy the inequality $\beta(\l)\le \beta_0(\l)$.

We are going to prove  that this inequality is strict if and only if the minimum  $\beta(\lambda)$ of the minimal surfaces free boundary problem \eqref{pbbeta}
is not reached by none of the trivial competitors  $u\equiv 0$ or $u\equiv 1$.
Accordingly, let us introduce the two following critical values of $\lambda$:
\begin{align}
\lambda_0 &= \sup \left\{ \lambda\ge 0 \ :\ u\equiv 0 \; \text{solves $\beta(\lambda)$}\right\},\label{def:lambda0} \\
\lambda_1 &= \inf  \left\{ \lambda\ge 0 \ :\ u\equiv 1 \; \text{solves $\beta(\lambda)$}\right\}.\label{def:lambda1}
\end{align}
After recalling the definition of the constants $h_D, \l_D, \theta_D$ given in \eqref{CheegerConst}, \eqref{def:lD} and \eqref{theta_O}, we  give here
a first result:
\begin{lemma} \label{intervals} Let $\l_0,\l_1$ be defined as above. Then
\begin{itemize}
\item [(i)] \ It holds $0\le\l_0 \le h_D$ and $u\equiv 0  $ solves $\beta(\lambda)$ if and only if $\l\in [0,\l_0]$.
\item [(ii)] \ It holds $\theta_D \l_D\le\l_1\le +\infty$ and   $u\equiv 1$  solves $\beta(\lambda)$ if and only if $\l\ge \l_1$.
\item [(iii)] \ The inequality  $\l_0\le\l_1$  is strict if $D$ is not self-Cheeger. 
\end{itemize}

\end{lemma}

\begin{proof} By taking $u\equiv 0$ as a competitor in \eqref{pbbeta}, we see that  $\beta(\l)\le |D|$ for any $\l\ge 0$ while the equality $\beta(\l)=|D|$ means that $u\equiv 0$
is a solution. Since the function $\beta$ is non increasing, it follows that $\beta(\l')=|D|$ for $\l'\in [0,\l_0]$. Therefore the subset appearing in \eqref{def:lambda0} is the full interval $[0,\l_0]$.
Futhermore, it holds  $\l_0\le h_D$ since $\beta_0(\l_0) \ge \beta(\l_0)=|D|$ implies that $m(\l_0, D) =0$ in virtue of \eqref{mlD=0}.
The positivity of $\l_0$ will follow from  Proposition \ref{estim2} where  a positive lower bound  is given.

Let us prove (ii). The set of $\l\ge 0$ for which  $u\equiv 1$ solves $\beta(\l)$ coincides  with the set 
$$J:=\{ \l\ge 0 : \beta(\l)=  |D| + P(D)-\l |D|\} .$$
 By the continuity of the concave function $\beta$, $J$ is closed and thereby  $\l_1\in J$.
It follows that $\one_D$ solves also $\beta_0(\l_1)$, hence  is a solution of $m(\l_1,D)$. In virtue of Corollary \ref{theta-min}, we 
infer the inequality $\l_1 \ge \theta_D\, \l_D$ .  Assume that $\l_1<+\infty$ and let $\l >\l_1$. 
To show that $J=[\l_1,+\infty)$, we need to demonstrate that $u\equiv 1$ is solution of $\beta(\l)$ for any $\l>\l_1$. This is a consequence of the following general comparison principle, namely that $\l \ge \mu$ implies that $u\ge v$ whenever $u, v$ solve $\beta(\l)$ and $\beta(\l\ge \mu)$. 
The conclusion will follow by applying it with $v\equiv 1$ solution to $\beta(\l_1)$ (taking into account that $u\le 1$).
Let us validate this principle in our case showing that $u\le 1$. Since $u$ minimizes  $E_\l$ (see \eqref{El}) on $BV(D;[0,1])$,
we have $E_\l(u) \le E_\l(\one_D)$, so that calling $F$ the functional which agrees with $E_\l$ when $\l=0$,  we get :
$  F(u)-F(\one_D) \le - \l \, | \{u<1\}| .$
Similarly as $\one_D$ minimizes $E_{\l_1}$, we obtain
$  F(\one_D) - F(u) \le   \l_1 \, | \{u<1\}|.$
Adding these two  inequalities above, we are led to\  $0\le (\l_1-\l)\,  |\{u<1\}|$, thus $|\{u<1\}|=0$. 
That proves that  $u=\one_ D$ is the unique minimizer of $\beta(\l)$ for every $\l>\l_1$.  

Finally, from the assertions (i) and (ii) and recalling that $\theta_D \ge 1$, we know  that 
$\l_0 \le h_D \le \l_D\, \theta_D\le \l_1$. Thefore, the inequality $\l_0\le \l_1$ is an equality if and only if $h_D=\lambda_D$ (and $\theta_D=1$).
This happens only if $D$ is self-Cheeger, hence the assertion (iii).

\end{proof}

Next we  state that, if $\lambda_0 <\lambda_1$ , then a solution to problem $\beta(\lambda)$ when $\l\in(\l_0,\l_1)$ will never be of the kind $\one_\Omega$. By the assertion (iii) of Lemma \ref{intervals}, this occurs in particular when $D$ is not a self-Cheeger set. 

\begin{theorem}\label{beta<beta0} Assume that  $\l_0,\l_1$ defined in \eqref{def:lambda0}\eqref{def:lambda1} 
 such that $\l_0<\l_1$. Then the strict inequality  
$\beta(\lambda) < \beta_0(\lambda)$ holds for every $\lambda\in (\lambda_0,\lambda_1)$.
\end{theorem}

\begin{proof}
Fix $\lambda\in(\lambda_0,\lambda_1)$. We shall prove the result by contradiction
assuming  that $\beta(\lambda) = \beta_0(\lambda)$. Then, $\beta(\l)$ admits a solution of the kind $ u=\one_\Omega$  where $\Omega \subset D$ is such that $0 <|\Omega| < |D|$.  For such a set of bounded perimeter $\Omega$, denoting  by $\partial \O$ its reduced boundary, the free boundary has positive measure namely
\begin{equation}\label{FB-exists}
\HH^{N-1}(\partial \O \cap D) >0  .
\end{equation}
  Indeed, if \eqref{FB-exists} where not true, then in virtue of \eqref{decompo}, we would have $D \one_\O=0$  in the distributional sense in the domain $D$, thus implying that $\one_\O\equiv 1$. Next we consider the trace $\trace(\one_\O)$ of $\one_\O$ on $\partial D$. As an element of $L^1(\partial D,[0,1])$, it is also the trace of a suitable function $\f\in W^{1,1}(D)$, in virtue of the Lipschitz regularity of $\partial D$ and of the fact that the trace operator 
  from $W^{1,1} \to L^1(D)$ is surjective by Galiardo's Theorem  (see \cite[Theorem 9]{gagliardo}). Without loss of generality, we can assume that $\f\ge 0$
   and possibly after adding to $\f$ the distance function to $\partial D$ (which belongs to $W_0^{1,1}(D)$) and after truncating the values greater than $1$,
   we build a non negative function $\f: D \to (0,1]$ such that $\trace f = 0$ on $\partial D\setminus \partial \O$ and $\trace f = 1$ on $\partial D\cap\partial \O$.
  Accordingly, for every small $\e>0$, we define $u_\e:= \max\{ \e \, \f, \one_\O\}$. For small $\e$, the level set  $\{u_\e\ge 1\}$ coincides with $\O$. On the other hand  the  distributional gradient of $u_\e$ in the open set $D$ decomposes as
   $ Du_\e = \e \, \nabla\f \, \LL^N\res \O  - (1- \e \f) \nu_\O \HH^{N-1}\res \partial(\O\cap D)$. Therefore, its total energy  is given by:
  $$  E_\l(u_\e) = \int_{D\setminus\O} \sqrt{1+ \e^2 |\nabla \f|^2} \, dx + \int_{\partial\O} |u_\e| \, d\HH^{N-1} + \int_{\partial\O\cap D}
   (1- \e \f)\, d\HH^{N-1}  - \l \, |\O| $$
  Since $\one_\O$ is a minimizer which shares  the same trace as $u_\e$ on $\partial D$, the following limit is non negative:
  \begin{align*}
    \lim_{\e\to 0^+} \frac{ E_\l(u_\e)-E_\l(\one_\O)}{\e} &= \lim_{\e\to 0^+} \int_{D\setminus\O} \frac{\sqrt{1+ \e^2 |\nabla \f|^2}-1}{\e} \, dx 
 - \,  \int_{\partial\O\cap D} \f\,   d\HH^{N-1} \\
 &=  - \,  \int_{\partial\O\cap D} \f\,   d\HH^{N-1} ,
 \end{align*}
where to pass from the first to the second line, we used dominated convergence. 
Recalling \eqref{FB-exists} and that $\f>0$ in $D$ by the previous construction, we infer that the limit above is strictly negative hence  the wished contradiction.

\begin{remark} \label{slope}  In the case where $\l_0\le 1< \l_1$,  it turns out that for $\l\in (1,\l_1)$,  a solution $u$ for $\beta(\l)$ (in the relaxed form \eqref{betarelax}) can't be in $W^{1,1}(D)$. Indeed in this case, the set $\O=\{u=1\}$ has a non empty free boundary $\partial\O \cap D$
and, by computing the  shape derivative of the  functional $J(\O)$ defined in \eqref{shape} (see for instance  \cite{boufraluc,henrot}), we obtain the optimality condition 
$1-\l= \frac{1}{\sqrt{1+|\nabla u|^2}}$ holding on $\partial\O \cap D$. This relation assigns the angle of the minimal surface with the plateau $\O \times {1}$.
Clearly this relation is  neither be fullfiled if $\l>1$. This means that any  $u$ solving  $\beta(\l)$ should exhibit a jump on the free boundary in order to reach the value $1$.

\end{remark}

\end{proof}

\medskip

\begin{example}\label{ex1D} In the one dimensional case, it is possible to compute explicitly the values of $\l_0,\l_1$. 
Owing to Theorem \ref{beta<beta0}, the occurence of the strict inequality $\beta(\l)<\beta_0(\l)$ is possible only if $\l_0<\l_1$. 
Without loss of generality, lets us consider  the domain  $D_R=(-R,R)$. It is a self-Cheeger set with constant $h_{D_R}=\l_{D_R}= \frac1{R}$
(and $\theta_{D_R}=1$). Recall that 
 \begin{align*}
\beta(\lambda)&= \inf\aco{ \int_{-R}^R \sqrt{1+u'^2}\, dx - \lambda|\{u\ge 1\}| \sepa u\in W^{1,1}(-R,R),\; u(\pm R) = 0},\\
\beta_0(\lambda) &= 2 R + m(\l, D_R) ,
\end{align*}
where the second equality follows from \eqref{ineqbeta}.
As shown in the next result, the strict inequality $\l_0<\l_1$  
 will occur if and only the length of interval $D$ is greater than $2$. 

\begin{lemma}\label{1dcase} Let $\l_0=\l_0(R)$ and $\l_1=\l_1(R)$ be the critical values  associated to  $D_R$. Then:
\begin{align} \label{1dcase-1} \l_0(R)\ =\ \begin{cases} \frac1{R} & \text{if $R\le 1$} \\ \frac2{1+R^2} & \text{if $R\ge 1$} \end{cases}  \quad,\quad
 \l_1(R)\ =\ \begin{cases} \frac1{R} & \text{if $R\le 1$} \\ 1 & \text{if $R\ge 1$} \end{cases}
\end{align}
\end{lemma}

\begin{proof}
Since $D_R$ is a Cheeger set with constant $\frac1{R}$ , we know that $m(\l,D_R) =0$  if $\l< \frac1{R}$,  while $m(\l,D_R)= 2(1 -\l R)$ otherwise. It follows that:
\begin{equation}\label{beta0-ex1}
\beta_0(\lambda) = 2R -2 (\l R-1)_+ 
\end{equation}
 Concerning the minimal value $\beta(\lambda)$, it is easy to check that it is achieved by taking $u$  to be either $u\equiv 0$ (then $\beta(\l)=2R$)
 or a radial function with plateau $\{u=1\}=\{|x|\le R-\a\}$  for a suitable value  $\alpha\in[0,R)$, of the form:   
 $$u_\alpha(x) = \min \left\{1, \frac{R-|x|}{\a} \right\}  \quad \text{if $\a>0$}\quad,\quad  u_0(x) \equiv 1.$$
 Observe that the expected solution $u_\a$ is always continuous. In view of the expression of $E_\l$ given in \eqref{El}, we obtain  the equality
 \begin{align}\label{def:f}
\beta(\l) = \min \left\{ 2R, \min_{\a\in [0,R]} f(\a) \right\} \quad \text{where}\quad f(\alpha): = 2 \pare{ \sqrt{1+\alpha^2} + (1\minus \l) (R\minus\a)}.
\end{align} 
Observe that the  function $f$ above is $C^1$, strictly convex with derivative
 $$f'(\alpha)= 2\pare{\frac{\alpha}{\sqrt{1+\alpha^2}}\minus(1\minus\lambda)}.$$
Moreover, the equality $\beta(\l)=2R$ means that $u\equiv 0$ is a minimizer, while if $\beta(\l)>2R$  the unique minimizer is
$u_\a$ where $\a$ is minimal for $f$ on $[0,R]$. According to the value of $R$, we will proceed in two cases.

\med {\bf Case $R\le 1.$\ }  If $\l \le \frac1{R}$,  we have $\beta_0(\l)=2R$ by \eqref{beta0-ex1}. If $\l\le 1$, we see directly  from the expression of $f$ in \eqref{def:f} that $ \inf f \ge 2 \ge 2R$.  Hence we have  $\beta(\l)=\beta_0(\l)=2R$ while  $u\equiv 0$ is a  minimizer for both problems.
If $\l > \frac1{R}$, then as $\l>1$, we have  $f'\ge 0$ and $\min f=f(0)=2 + 2(1-\l) R <2R $.
Therefore, an optimal solution is given by $u=u_0\equiv 1$ and $\beta(\l)= \beta_0(\l)$ in view of \eqref{beta0-ex1}.
In conclusion, for every $\l\ge 0$, we have $\beta_0(\l)= \beta(\l)$ and the the critical  values $\l_0(R)$ and $\l_1(R)$ are equal to 
$ \frac1{R}.$ Note that, for $\l=\frac1{R}$, the solutions $u\equiv 0$ and $u\equiv 1$ coexist.

\med {\bf Case $R> 1.$\ } 
If $\l\ge 1$, as noticed before, $f'\ge 0$ and $\min f= f(0)<2R$.
Hence $\beta(\l)=\beta_0(\l)$ and the unique common minimizer is $u\equiv 1$. 
If $\l<1$, the convex function $f$ starts with a negative slope at zero so that it reaches its mimimum at a unique  $ \a_c\in (0,R]$ provided
 $f'(R) \ge 0$, that is if $\l \ge \l_* = 1 -\frac{R}{\sqrt{1+R^2}}$. 
 In view of \eqref{def:f},   $u_{\a_c}$ will be the unique minimizer of $\beta(\l)$
 if, in addition, it holds $f(\a_c)\le 2R$. In this case, since  $u_{\a_c}$ is  not of the form $\one_\O$,
we  will deduce the strict inequality $\beta(\l) <\beta_0(\l)$.
After some computations  \footnote{Setting $t=1-\l$ and $h(x) := \sqrt{1+x^2} - t x + R(t-1)$, we are reduced  to 
show that $min_{[0,R]} h <0$ iff $0 < t < \frac{R^2 - 1}{R^2 + 1}$. The minimum of $h$ is reached
at $x_c$ such that $\frac{x_c}{\sqrt{1+x_c^2}} = t$. 
Since $x_c = \frac{t}{\sqrt{1-t^2}}$ while $ \sqrt{1+x_c^2} = \frac{1}{\sqrt{1-t^2}}$, we obtain:
\begin{align*} \min h = \sqrt{1+x_c^2} - t x_c + R(t-1)=
 \frac{1}{\sqrt{1-t^2}} - t \left( \frac{t}{\sqrt{1-t^2}} \right) + R(t-1) = \sqrt{1-t^2} + R(t-1)
\end{align*} 
Recalling that $t<1$, we have $\min h<0$ iff $\sqrt{1-t^2} < R(1-t)$. Squaring and  dividing by $1-t$,
we are led to the condition  $1+t < R^2(1-t)$, that is $ t < \frac{R^2 - 1}{R^2 + 1}$.
 (note that $\l^*\in [\l_*,1]$ ensures that $\a_c$ belongs to $(0,R]$).}, it turns out that  the equality   $f(\a_c)\le 2R$ is true iff $\l \ge \l^*$ where $\l^*:=\frac2{1+R^2}$ If $\l\in [\l_* , \l^*)$, the minimum of $f$ is larger  than $2R$. It is also the case if 
 $\l< \l_*$, since $f'>0$ on $[0,R] $ implies that  $\min f=f(0)> 2  \ge 2R$.
 Therefore, the minimum of  $\beta(\l)$ is reached at $u\equiv 0$ for $\l\le \l^*$. For $\l \in (\l^*,1)$, there is a unique 
 solution of the kind $u_{\a}$ with $\a>0$, and ultimately,  the solution $u\equiv 1$ for $\l\ge 1$.  
    Summarizing, we have proved that $\l_0(R) =\frac2{1+R^2}$ and $\l_1(R)=1$
 \end{proof}

\end{example}

%
%
%

\med

In higher dimension $N\ge 2$ , explicit expressions for $\l_0$ and $\l_1$ are not available, except possibly in the radial case.
However we are able to derive some estimates where the role of the geometric constants   $h_D, \l_D$ and $\theta_D$ is enlightened. 
From Lemma \ref{intervals}, we already know that $\l_0 \le h_D \le \theta_D\, \lambda_D \le \l_1$.
In the next result, we use the duality result presented in Subsection \ref{dualb0} for deriving a positive lower bound for $\l_0$ and a sharp upper bound for $\l_1$ when  $D$ is calibrable.

\begin{proposition}\label{estim2} Let $D\subset \R^N$ be a general bounded Lipschitz domain. Then the critical values $\l_0,\l_1$ defined in \eqref{def:lambda0}\eqref{def:lambda1} satisfy
:
\begin{itemize}
\item[(i)] $\lambda_0 \ge \lambda_0^* $  where 
\begin{align}
\lambda_0^* := \begin{cases}
 1- \cos(h_D) & \text{if $h_D\le \frac{\pi}{2}$}\\
1 + h_D-\frac{\pi}{2} & \text{if $h_D\ge \frac{\pi}{2}$}
\end{cases}\; . \label{l0star}
\end{align}
\item[(ii)] Let $\theta_D$ be the (posssibly infinite) calibration constant of $D$ and $\lambda_D :=P(D)/|D|$. Then
\begin{align}
\theta_D\lambda_D\ \le\ \lambda_1\ \le\ \theta_D\lambda_D + 1.\label{l1esti}
\end{align}
\end{itemize}
\end{proposition}

\begin{remark}\label{convexD} If D is a convex set of class $C^{1,1}$, then, by virtue of Proposition \ref{chathetacalib}, the inequalities \eqref{l1esti} can be rewritten as follows:\begin{align}
\max \{\lambda_D, (N-1)\kappa_\infty(\partial D)\} \le  \lambda_1 \le 1+ \max \{\lambda_D, (N-1)\kappa_\infty(\partial D)\}. \label{l1estim}
\end{align}
In particular, if  $\partial D$ exhibits a corner, then $\theta_D=+\infty$  and $\lambda_1=+\infty$. This means that  $u\equiv 1$ can't neither be a solution to  $\beta(\lambda)$.

Note that in the one dimensional case, we have $h_D=\l_D = |D|^{-1}$ and $\theta_D=1$ so that \eqref{l1esti} becomes $\l_D\le \l_1\le 1+\l_D$. 
In Lemma \ref{1dcase}, we showed that if $|D|=2R$ with $R>1$ , then $\l_1=1$ so that $\l_1 \in [\frac1{R}, 1 + \frac1{R}]$. By sending $R$ to $1$ or to $+\infty$,
we conclude that the bounds in \eqref {l1esti} are optimal. 
\end{remark}


\begin{proof} We will construct calibration fields $\sigma \in L^\infty(Q; \R^N\times \R)$ where $Q:=D\times [0,1]$ of the kind
\begin{align} \label{specialsig}
\sigma(x,t)\ =\ \Big( -a(t)q(x),\; A(t)\Div q(x) + r(x) \Big) \quad \text{for } (x,t)\in Q
\end{align}
where
\begin{equation}\label{condiA}
\left\{
 \begin{array}{ll}
 &   a\in C([0,1])  \ \text{is such that}\  0 \le a(t)\le 1  ; \\
 &   (q,r)\in L^\infty(D;\R^N\times \R),  \  |q| \le 1\ ,\   \Div q\in L^\infty(D) ;\\
 &   A'(t) = a(t)  \ \text{(thus $\Div \sigma =0$ in $Q$).} 
\end{array}
\right.
\end{equation}
Note that, with the last condition, $A$ is Lipschitz non decreasing with a slope $A't)\le 1$.
%


\noindent
\emph{Proof of the assertion (i): }\ As noticed in the proof of Lemma \ref{intervals}, it holds $\l\le \l_0$ id and only if $\beta(\l)\ge |D|$.  In view of the duality Theorem \ref{duality2}, 
it will be the case if we can find an admissible $\sigma=(\sigma^x,\sigma^t)$ such that 
\begin{gather}\label{CS-sigma}
 \sigma^t(x,0) = -1\, , \, \sigma^t(x,1) \ge \lambda -1 \, \tinm{on} D \, , \, \sigma^t + \sqrt{1- |\sigma^x|^2} \ge 0 ,
 \,  \tinm{in} Q.
\end{gather}
Let us search $\sigma$ of the form given in \eqref{specialsig} where $r\equiv -1$ and the triple $(a,A,q)$ satisfies \eqref{condiA}. We impose the additional  condition that $q$ satisfies $\Div q=h_D$. By  Remark \ref{1calib}, such a $q$ exists and it calibrates every  Cheeger set of $D$. With this choice, we obtain
 that $\sigma^t + \sqrt{1- |\sigma^x|^2} \ge h_D\, A(t)  + \sqrt{1- A'^2(t) }$. It follows that
 the conditions in \eqref{CS-sigma}  are all met if we  select $A(t)$ so that
$$ \,  A(0)=0 \, ,\quad 0\le A'(t)\le 1 ,  \quad  h_D\, A(1)  \ge \l \, , \quad  h_D\, A(t)  + \sqrt{1- A'^2(t) } \ge 1 \quad \forall t \in[0,1].$$
Note  that the conditions above imply that   $A(1) \le 1$. Thus   the inequality $h_D\, A(1)  \ge \l$ can't be reached unless $\l \le h_D$, 
which is coherent with the assertion (i) of Lemma \ref{intervals}. For further computations, it is convenient to set
 $\psi(t):=h_D A(t)$.  Then we arrive to the fact that  $\l\le \l_0$ whenever
\begin{align}
\lambda \le  \sup \left\{ \psi(1) :\,  \psi(0 ) = 0, \ 0\le \psi'\le h_D  \ \text{and}\ \sqrt{1-\frac{|\psi'|^2}{h_D^2}} + \psi \ge 1 \ \text{on [0,1]}\right\}.\label{maxpsi}
\end{align}
Hence,  proving the assertion (i) reduces to check that the right hand side of \eqref{maxpsi} coincides with the value $\l_0^*$ given by \eqref{l0star}.
In fact the inequality constraint on $\psi$ in \eqref{maxpsi} can be rewritten equivalently as 
$0\le \psi' \le h_D \gamma(t)$ where
\begin{align*}
\gamma(s): =
\begin{cases}
\sqrt{s(2-s)} &\text{if } s \le 1\\
1 &\text{if } s>1
\end{cases}
\end{align*}
It follows that the composed function $z: [0,1]\mapsto [0,h_D]$ defined by 
$ z(t) := \int_0^{\psi(t)} \frac{1}{\gamma(s)} ds $
satisfies $0\le z'(t) \le h_D$ with a maximal value $h_D$ reached at $t=1$ if and only if $z(t)=h_D t$. Accordingly the maximal value in \eqref{maxpsi} is reached
for $\psi$ determined by the following relation holding for every $t\in [0,1]$: 
$$h_D t\ =\ \int_0^{\psi(t)} \frac{1}{\gamma(s)} ds\ =\ \begin{cases}
\arccos (1-\psi(t)) &\text{if }\, \psi(t)\le 1,\\
\frac{\pi}{2} + \psi(t)-1  &\text{if }\, \psi(t)>1.
\end{cases}. $$
We conclude that  the optimal $\psi$ for \eqref{maxpsi} is given explicitely by
$$ \psi(t) =\begin{cases}
 1- \cos(h_D t) & \text{if $h_D t\le \frac{\pi}{2}$}\\
1 + h_D t -\frac{\pi}{2} & \text{if $h_D t\ge \frac{\pi}{2}$}
\end{cases}
$$ 
thus confirming the optimal value  $\psi(1)=\l_0^*$. 

\medskip\noindent
ii) The inequality $\theta_D\lambda_D \le \lambda_1$ has been proved in Lemma   \ref{intervals}. 
In order to show that $\lambda_1 \le 1+ \theta_D\lambda_D$,  we need to demonstrate that $\beta(\l)\ge\beta_0(\l)$ for any  $\lambda > 1+ \theta_D\lambda_D$.
 Let us fix such a $\l$. Without any loss of generality, we can assume that $\theta_D<+\infty$ since otherwise $\l_1=+\infty$.
 Then, by applying Corollary \ref{theta-min} to $\O=D$, we see that $D$ is minimal for  $m(\lambda-1,D)$. Then, in virtue to Theorem \ref{duality2}, it exists a calibrating field $q$ satisfying
\begin{align*}
|q|\le 1, \quad
0\le \Div q \le \lambda-1 \tinm{in} D,\quad\tinm{and}\quad q\cdot \nu_D =1 \tinm{on} \partial D.
\end{align*}
Next we consider a vector field $\sigma$ of the form \eqref{specialsig} where $a(t)\equiv 1$, $A(t)=t-1$ and $r(x)\equiv \l-1$, that is:
\begin{align*}
\sigma(x,t)= ( -q(x), (t-1)\Div q(x) +\lambda -1 ).
\end{align*}
It is easy to verify that $\sigma$ is admissible for the dual problem of $\beta(\lambda)$, since by construction:
\begin{align*}
\sigma^t(x,1) = \lambda -1  \tinm{in} D, \quad \Div\sigma =0 \tinm{in} Q, 
\end{align*}
while by the inequalities $0\le \Div q \le \l-1$, we have  for a.e.  $(x,t)\in Q$
\begin{align*}
\sqrt{1-|\sigma^x(x,t)|^2} +\sigma^t(x,t) \ge \sigma^t (x,t) = (t-1)\Div q(x) +\lambda -1   \ge 0.
\end{align*}
Therefore,  in view of the duality Theorem \ref{duality2} ,   it holds
\begin{align*}
\beta(\lambda) \ge -\int_D \sigma^t(x,0)= (1-\lambda)|D| + \int_D \Div q \, dx=  (1-\lambda)|D| + P(D)   =\beta_0(\lambda) ,
\end{align*}
where: 
\begin{enumerate}
\item [-] for the first equality, we used the fact that $\int_D \Div q \, dx = \int_{\partial D} q\cdot \nu_D\, d\HH^{N-1}$
\item [-] for the last equality, we used the fact that $D$ is optimal for $m(\mu,D)$  for any $\mu\ge \l-1$, hence for $m(\l,D)$ in particular.
\end{enumerate}

We conclude  that $\beta(\lambda)=\beta_0(\lambda)$, that means $u\equiv 1$ solves $\beta(\lambda)$, whence $\lambda \le \lambda_1$.

\end{proof}

\begin{example}
Let $D=\{x\in\RR^2 :  |x| < R \}$ be a disk of radius $R$ in $\RR^2$. Recall that disks are self-Cheeger sets and that the Cheeger constant of a disk is completely determined by its radius, i.e. $h_D=2/R$. The unique solution to problem $\beta_0(\lambda)$ is $u_0\equiv 0$ for $\lambda < h_\Omega$ where $h_\Omega= \frac2{R}$ is the Cheeger constant of $\Omega$.
In contrast the unique solution is  $u_1\equiv 1$ for $\lambda>h_\Omega$. For the precise value $\lambda= h_\Omega$ we obtain exactly two solutions $u_0, u_1$. 
Accordingly,
\begin{align*}
\beta_0(\lambda) =
\begin{cases}
|D| &\tinm{if} \lambda \le h_D \\
P(D) + (1-\lambda)|D| &\tinm{if} \lambda >h_D.
\end{cases}
\end{align*}

Let us now turn to the determination of the infimum $\beta(\lambda)$. 
By a rearrangement argument, we can prove that solutions are all radial of the form $u(x) = \f(\frac{|x|}{R})$ being $\f(t):[0,1]\to [0,1]$ monotone non increasing.
The plateau $\{u=1\}$ is associated with an interval $t\in [0, \rho]$ for a suitable value of $\rho\in [0, 1]$ to be determined. 
For such a plateau, the minimal surface problem reads
\begin{align}\label{minsurf}
 J(\rho) :=\inf_{\substack{\f(\rho) = 1\\ \f(1) = 0}} I(\f)  \ ,\qquad I(\f) := R \int_\rho^1 \sqrt{R^2+\f'^2} \, tdt \ . 
\end{align}
The first integral of Euler equation for this minimization problem reads
\begin{align}
\frac{t\f'}{\sqrt{R^2 + \f'^2}} = \mu \label{int1Euler}
\end{align}
for some constant $\mu$. As $\f(1) = 0$, we are led to the explicit form 
\begin{align}
\f (t) = K(\mu,t)  \;,\qquad K(\mu,t) := \mu \, R \, \log \left( \frac{1 + \sqrt{1 - \mu^2}}{t + \sqrt{t^2 - \mu^2}} \right)\;,\label{sol_phi}
\end{align}
provided we can find $\mu \in [0,\rho]$ such that $\f(\rho)=K(\mu,\rho)=1$.
In fact, since the function $\mu \mapsto K(\mu,\rho)$ is stricly increasing on $[0, \rho]$,   such a $\mu$ exists and is unique and exists if and only if
\begin{align}
1 \,\le\,  K(\rho,\rho)= \rho\, R\, \log\left(\frac{1 + \sqrt{1 - \rho^2}}{\rho}\right). 
\label{exist_cond} 
\end{align}
If the inequality above is strict, then we obtain a unique  solution to \eqref{minsurf} which is smooth.  
In the limit case where \eqref{exist_cond} is an equality, one has $\mu=\rho$ and  $\f'(\rho)=+\infty$.
If $K(\rho,\rho)<1$, then  \eqref{minsurf} has no solution but the  relaxed solution in $BV(D)$ is unique and exhibits a jump at $t=\rho$
of amplitude $1- K(\rho,\rho)$.
In all cases, we have determined, in term of parameter  $\rho\in [0,1]$, an optimal radial configuration whose plateau $\{u=1\}$
agrees with the disk $B(0,\rho)$. Its total energy is given by
$$ E (\rho) :=  2\pi \, J(\rho) + (1-\lambda) \pi \rho^2 R^2 \ .$$
In order  to minimize $E(\rho)$ on interval $[0,1]$, we introduce 
\begin{align*}
\ov\mu(\rho):= \sup_{0\le \mu \le \rho} \{\mu : K(\mu,\rho) \le 1\}.
\end{align*}
It is easy to check that for every $\rho\in[0,1]$ such that if $K(\rho,\rho)< 1$ then $\ov\mu(\rho) =\rho$. Otherwise, $\ov\mu(\rho)$ is the unique solution of equation $K(\mu,\rho)=1$.
After a straightforward computation and exploiting \eqref{int1Euler}, we obtain
$$ J(\rho) = R^2\int_\rho^1 \frac{t^2}{\sqrt{t^2 - \ov\mu(\rho)^2}} dt + \rho R \Big(1- K(\ov\mu(\rho),\rho) \Big) \,.$$
Thus, noticing that $ K(\ov\mu(\rho),\rho)\le 1$, we are led to:
\begin{align}\label{Erho}
\scalebox{0.98}{$
E(\rho) = \pi R^2 \left( \sqrt{1 - \ov\mu(\rho)^2} - \rho\sqrt{\rho^2 - \ov\mu(\rho)^2} + \ov\mu(\rho)^2\log\frac{1+\sqrt{1 - \ov\mu(\rho)^2} }{\rho+\sqrt{\rho - \ov\mu(\rho)^2}} + \frac{2\rho \Big(1- K(\ov\mu(\rho),\rho) \Big) }{R}  + (1-\lambda)\rho^2   \right)
$}
\end{align}
Finally, we need to determine an optimal $\ov\rho$ for $\min\{ E(\rho)  : \rho \in [0,1]\}$.  
Then the  radial function $\ov u(x)=\ov \f(\frac{|x|}{R}) = K\left(\ov\mu(\ov\rho),\frac{|x|}{R}\right)$ defined in \eqref{sol_phi}  minimizes the relaxed problem \eqref{betarelax} . This solution is continuous if  $K(\ov\rho, \ov\rho)= 1$ and otherwise exhibits a jump of amplitude  $1-K(\ov\rho, \ov\rho)$ before reaching the value $1$ on the plateau. It turns out that the jump of $\ov u$ occurs when $\l$ passes the value $1$, thus confirming the behavior predicted in Remark \ref{slope}. This is illustrated in   Figure \ref{optU} 
 \begin{figure}[H]
\centering
\includegraphics[scale=1.]{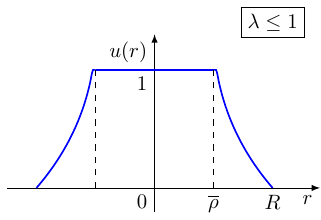}
~\hspace{2.em}
\includegraphics[scale=1.]{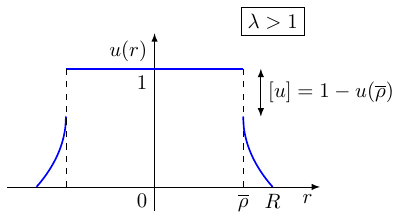}
\caption{Illustration for an optimal $u$ and optimality conditions.} \label{optU}
\end{figure}

The minimization of $E(\rho)$ is performed by using Matlab for different values of $R$ and $\lambda$. As predicted by Theorem \ref{beta<beta0},  $\beta_0(\lambda)$ coincides with $\beta(\lambda)$ outside the interval $(\lambda_0,\lambda_1)$ where their common minimizers are trivial characteristic functions either $\ov u \equiv 0$ ($\lambda \le \lambda_0$) or $\ov u \equiv 1$ ($\lambda \ge \lambda_1$). When $\lambda_0 < \lambda_1$, the strict inequality $\beta(\lambda) < \beta_0(\lambda)$ occurs for any $\l\in (\l_0,\l_1)$ and  the minimizer $\ov u$ for problem $\beta(\lambda)$ provides a true minimal surface with a possible jump under the plateau 
$\{\ov u=1\}$. The numerically computed critical values  $\lambda_0$, $\lambda_1$ are represented in term of $R$ in Figure \ref{figlambda}.
  We observe that $\l_0=\l_1$ for $R\le 1$ whereas the strict inequality $\l_0<\l_1$ occurs for all $R>1$.
      Since $D$ is a disk, we have $\theta_D=1$ and $\lambda_D = h_D = 2/R$. Then the bounds provided 
in Proposition \ref{estim2} become:
\begin{align*}
\lambda_0^*\le \lambda_0 \le \frac{2}{R} \le \lambda_1 \le 1+ \frac{2}{R}\quad\text{where}\quad \lambda_0^* =
\begin{cases}
1-\cos\pare{\frac{2}{R}} &\text{if } R \ge \frac{4}{\pi},\\
1+ \frac{2}{R} - \frac{\pi}{2} &\text{if } R < \frac{4}{\pi}.
\end{cases}
\end{align*}
In Figure \ref{figlambda}, these bounds are represented as functions of $R$. The curve representing  $\lambda_1$ suggests the following exact value
$ \lambda_1 = \max \aco{\frac{2}{R}, 1+ \frac{1}{R} }.$
\begin{figure}[H]
\centering
\includegraphics[scale=1.]{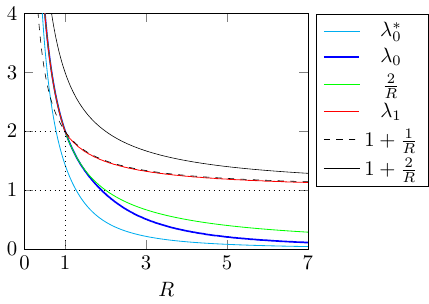}
\caption{Critical values of $\lambda$ in term of the radius $R$ of a disk  $D\subset\RR^2$.}\label{figlambda}
\end{figure}
Next we display in  Figure \ref{figRle1} (case $R\le 1$) and in Figure \ref{figRge1} (case $R>1$), the dependence upon $\l$ of
$\beta_0(\lambda), \beta(\l)$, of 
the jump $[u]$ of the radial solution and of the ratio $\frac{\ov \rho}{R}$ being $\ov \rho$ the radius of the plateau.

\begin{figure}[H]
\centering
\includegraphics[scale=0.8]{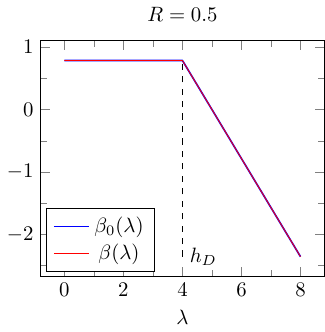}
\includegraphics[scale=0.8]{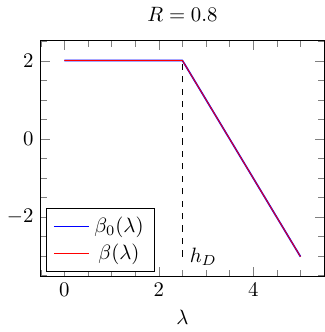}
\includegraphics[scale=0.8]{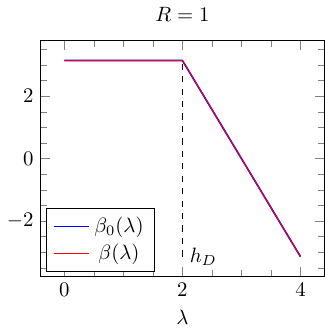}\\
\includegraphics[scale=0.8]{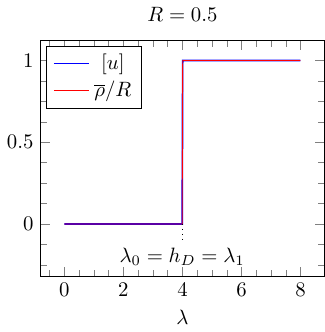}
\includegraphics[scale=0.8]{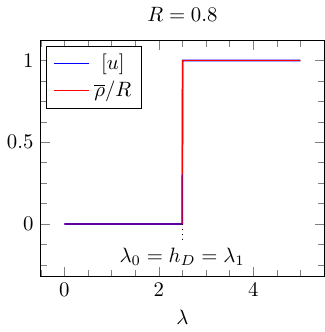}
\includegraphics[scale=0.8]{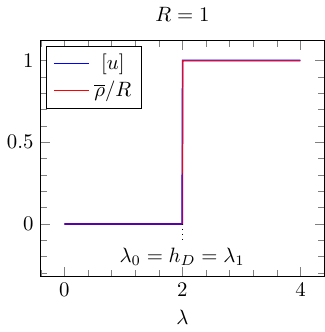}
\caption{Dependence in $\l$ for  $R \in \{0.5, 0.8, 1\}$. }\label{figRle1}
\end{figure}

\begin{figure}[H]
\centering
\includegraphics[scale=0.85]{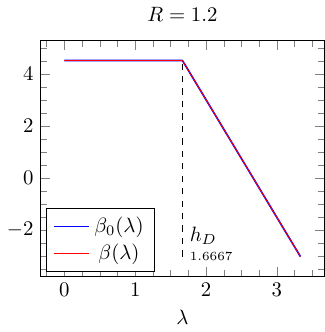}
\includegraphics[scale=0.85]{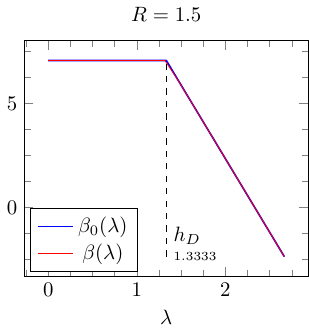}
\includegraphics[scale=0.85]{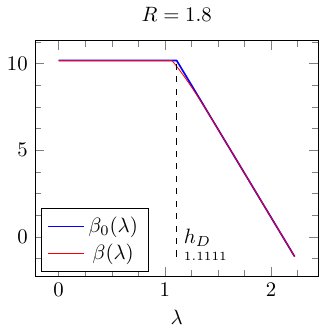}\\
\includegraphics[scale=0.85]{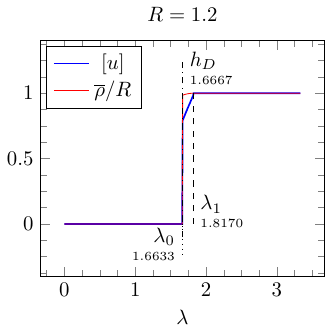}
\includegraphics[scale=0.85]{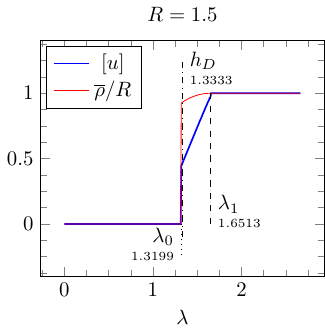}
\includegraphics[scale=0.85]{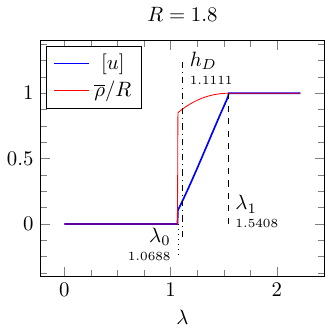}\\
\includegraphics[scale=0.85]{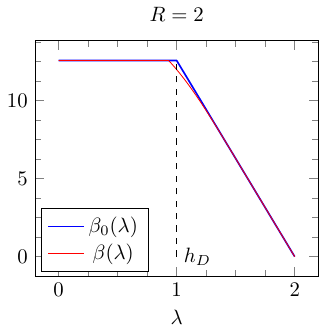}
\includegraphics[scale=0.85]{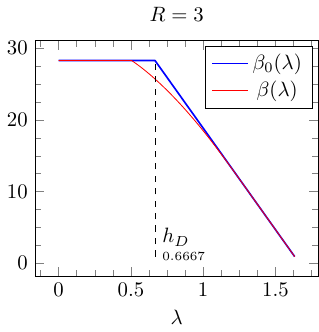}
\includegraphics[scale=0.85]{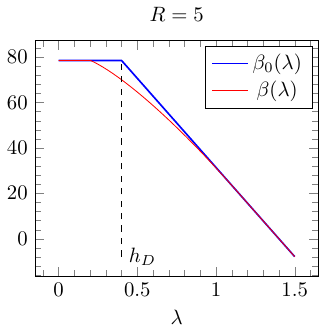}\\
\includegraphics[scale=0.85]{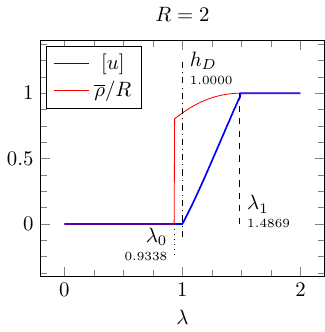}
\includegraphics[scale=0.85]{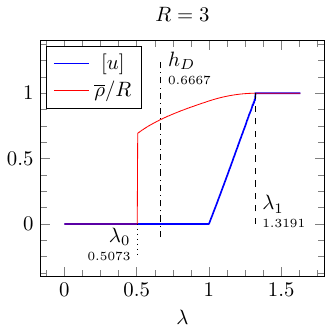}
\includegraphics[scale=0.85]{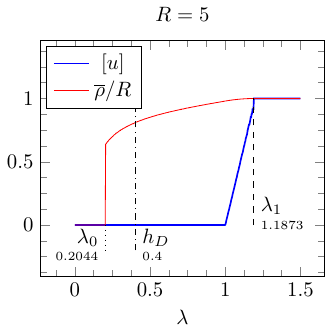}
\caption{Dependence in $\l$ for  $R \in \{1.2, 1.5, 1.8, 2,3,5\}$. }\label{figRge1}
\end{figure}

\end{example}

\medskip

\section{Cut-locus potential and construction of a two-dimensional calibration}
\label{scutlocus}

\markright{Cut-locus potential. An explicit construction of 2D calibrating fields}

\bigskip
In this section, we assume that\emph{ $D$ is a bounded convex open subset of $\RR^2$.}
The following  notations will be used:
\begin{enumerate}[leftmargin=6mm]
\item[-] $D^c$ denotes the complement of $D$ in $\R^2$.
\item[-] $B(x,r)$ denotes the open ball in $\RR^2$ centered at $x$ with radius $r>0$;
\item[-]  For two points $x,y\in \RR^2$, the segment joining them is denoted by $[x,y]:= \{ (1-t)x + ty \sepa t\in[0,1] \}$ ;
\item[-] For every  subset $A\subset\RR^2$, we write $A^c$ for the complement; if $A$ is non empty, $d(x,A)$ denotes the Euclidean distance from $x$ to $A$, namely
$d(x,A):= \inf \aco{ |x- y| \sepa y\in A}$; 
\item[-] If $A$ is a closed subset of $\RR^2$, we denote $\Pi_A(x):=\{y\in A \sepa |x-y| =d(x,A)\}$. In case $A$ is convex, $\Pi_A(x)$ is a singleton.

\item[-] The distance between two non empty subsets $A,B\subset \RR^2$ is given by 
$$d(A,B):=\inf \aco{ d(x,y) \sepa (x,y)\in  A\times B } .$$
\end{enumerate}

\subsection*{Normal cone to $D$ and regular points of $\partial D$.}\ Since $D$ is convex, the Gauss map $\n:\partial D^\delta \to S^1$ which assigns to  $x$   its normal $\nu_D(x)$ is well defined except in an at most countable number of points where we have left and right limits  denoted by $\nu_D^-(x)$ and $\nu_D^+(x)$. This singular set will be denoted $\partial_s D$ while its complement $\partial_r D$ is the set of regular points. 

The normal cone to $D$ at $x \in \ov D$ is defined by  
$$ N_{D}(x):= \aco{x^*\in\RR^2 \sepa \ps{x^*,y-x} \le 0,\; \forall y\in D }.$$
It reduces to $\{0\}$ if $x$ is inside $D$, while if $x\in \partial D$, it is generated by the unit vectors $\nu_D^-(x)$ and $\nu_D^+(x)$, that is $N_D(x):=\aco{ s\,\nu_D^+(x) + t\,\nu_D^-(x) \sepa s,t \in \RR_+ }$.
If $x\in\partial_r D$, then $\nu_D(x) =\nu_D^+(x) =\nu_D^-(x)$ and $N_D(x)= \RR_+ \nu_{D}(x)$ is a single positive ray.

\subsection*{Some important convex subsets of $D$. }
Let $R_D:= \max_{x\in D} d(x, D^c)$ be the inradius of $D$. For every $\delta$ such that $0\le \delta<R_D$, we introduce the set
\begin{align*}
D^\delta:= \{x\in D\ :\ d(x, D^c) >\delta \}.
\end{align*}
One checks easily that $D^\delta$ is a non void open convex subset 
\footnote{For any $x_1,x_2\in D^\delta$, we have $B(x_1,\delta), B(x_2,\delta)\subset D$. Thus, since $D$ is convex
\begin{align*}
B(\frac{x_1 + x_2}{2},\delta)= \frac1{2} B(x_1,\delta) + \frac1{2} B(x_2,\delta)\subset \frac1{2} (D+D)  =D \implies \frac{x_1 +x_2}{2} \in D^\delta.
\end{align*}
}. 
Acoordingly, we will denote by $\Pi_\delta(x)$ the orthogonal projection on $\ov {D\delta}$ of any $x\in\R^2$. 

\med
Next, for every $\lambda \ge R_D^{-1}$, we consider the $\delta$-enlargement of $D^\delta$ when $\delta= \l^{-1}$.  This  convex open subset of $D$ will play a crucial role in what follows. It is given by the union of all balls of radius $\lambda^{-1}$ contained in $D$, namely
\begin{align}\label{def:Olambda}
\Omega_\lambda = \bigcup_{ B(x,\lambda^{-1}) \subset D} B(x,\lambda^{-1}).
\end{align}
For $\delta= R_D$, we obtain the \emph {maximal balls inscribed subset} of $D$ defined by:   
$$U_D:=\Omega_{\frac{1}{R_D}} \ =\  \{ x\in D \ :\ d(x, \Sigma_D) < R_D \} ,$$
 being  $\Sigma_D :=\{ x\in D\ :\ d(x, D^c) = R_D)\}$ (the \emph{ high ridge} of $D$).

\medskip
In this Section, our goal  is to show that, for every $\l\in [h_D,+\infty)$, the set $\O_\l$ defined in \eqref{def:Olambda} is calibrable and optimal for $m(\l,D)$.
To that aim, we will use a geometrical construction for defining a locally Lipschitz potential $\rho: \ov D \to (0,+\infty)$ which has its own interest. 
Before that let us recall the two-dimensional construction of  the Cheeger set of a convex set $D$ following \cite[Theorem 3.32 i)]{stredulinsky1997area} and the celebrated result of Bernd Kawohl and Thomas Lachand-Robert:
\begin{theorem}[\cite{kawohl2006characterization}, Theorem 1]\label{chaCheeger2}
There exists a unique value $\delta=\delta^*$ such that $|D^\delta| = \pi\delta^2$. Then, $h_D=1/\delta^*$ and the Cheeger set of $D$ is $C_D = \cup \{B(x,\delta^*):B(x,\delta^*) \subset D\}$.
\end{theorem}
The latter result says that $D$  admits a unique Cheeger set given by $C_D=\O_{h_D}$. Moreover,  the $1$-calibrability of $D$  (as defined in Subsection \ref{calibrable}) is equivalent to the fact that $D=\O_{h_D}$
(hence, by Corollary \ref{theta-min}, to the curvature upper bound $\kappa_\infty(\partial\O)\le h_D$). 

\subsection{Cut-locus potential}
\label{sec_cutlocus}

%

We introduce the function $\rho=  \overline{D} \to \RR_+$ defined by
\begin{align}\label{def:rho}
\rho_D(x) := \sup\{ \delta\ge 0 \ : \ d(x,D^\delta) \le \delta\}.
\end{align}
Since this potential $\rho_D$ will not be used for another domain, we will simply write $\rho_D$ as $\rho$. 
A good reason to call it \emph{cut-locus potential} is that $\rho$ is a continuous extension of  the normal distance to the cut-locus of $D$ 
 defined on $\partial D$ by
\begin{align}\label{cut-dis} 
\tau(x):= \begin{cases}
\sup \{t\ge 0 \sepa  x= \Pi_{\partial D} (x-t \nu_D(x))  \}\quad & \text{ if $x\in \partial_r D$} \\
0 & \text{ if $x\in \partial_s D$}
\end{cases}
\end{align} 
The cut-locus of $D$ is the closure of the  singular set $\Lambda_D$ of $d(\cdot,D^c)$ 
that is the set of point $x\in D$ where $d(\cdot, D^c)$ is not differentiable, that is
$$\Lambda_D:=\{x\in D \sepa \Pi_{D^c}(x) \ \text {is not a singleton}\}.$$ 
 We refer to \cite{crasta2016cutlocus, crasta2016characterization, crasta2017geometric, crasta2015dirichlet} for more details on this notion.

\medskip
As it is proved in  Lemma \ref{alpha}, for every  $x\notin \ov{U_D}$, we have the equality
 $$\{\delta\ge 0 \sepa d(x,D^\delta)\le \delta \} = [0,\rho(x)].$$
This property is illustrated  in Figure \ref{falpha}, where
 the function  $\delta \to \alpha(x,\delta) := d(x,D^\delta) - \delta$ is  negative on the interval $(0,\gamma(x)]$ and then
 is stricly increasing reaching a positive value at $x=R$.
 Therefore  $\rho(x)$ is characterized as the unique zero of $\alpha(x,\cdot)$ on $(0,R)$ if $x\notin \ov{U_D}$.
The lower-bound inequality $\rho(x) \ge \gamma(x)$ which is is strict if $x\in D$  involves  the
distance from $\partial D$ to the cut-locus $\Lambda_D$  along the normal to $\partial D$ passing through $x$ , precisely:
 \begin{align}\label{ndis}
\gamma : x\in \ov D \mapsto  \zeta(x)+ d(x,D^c)\ ,\ \zeta(x):=
\begin{cases}
\min\{t\ge 0: x + t \nabla d(x,D^c) \in \ov\Lambda_D \} \ &\forall x\notin  \Lambda_D\\
0 &\forall x\in \Lambda_D
\end{cases}
\end{align}
Note , in particular, that  $\zeta=0$ and $\gamma(x)= d(x,D^c)$ for every $x\in \ov \Lambda_D$. On the other hand,  we infer from \eqref{cut-dis}, that 
\begin{equation}\label{gamma=tau}
\gamma(x)=\, \zeta(x)=\, \tau(x) \quad \forall x\in \partial D.
\end{equation}
 Summarizing , our potential $\rho$ can be characterized as follows:
\begin{equation}\label{master-rho}
 \rho(x) = d(x, D^{\rho(x)}) \ \text{if $x\in D\setminus \ov{U_D}$}\ ,\
 \rho(x)= \tau(x)\ \text{if $x\in \partial D$}\ ,\   \rho(x) = R_D \ \text{if $x\in\ov{U_D}$}
\end{equation}

\begin{figure}[H]
\centering
\includegraphics[scale=1.]{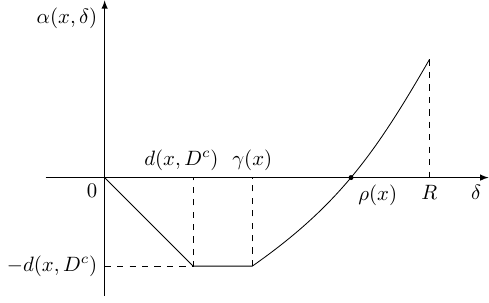}
\caption{Function $\alpha(x,\cdot)$ for $x\notin \ov{U_D}$.} \label{falpha}
\end{figure}

\begin{lemma}\label{lemrho1}
The potential $\rho$  reachs its maximum $R_D$ on the plateau $\ov{U_D}$ and satisfies the inequality
$\rho(x)\ge d(x,D^c)$.
Moreover, recalling the definition of $\O_\l$ in  \eqref{def:Olambda},  we have   the equality:
\begin{align} \Omega_\lambda = \aco{x\in D \sepa \rho(x) > \frac1{\l}}\ ,\ \text{for every $ \l \ge \frac1{R_D}$} .\label{equi2}
\end{align}

\end{lemma}

\begin{proof} For $x\in D$, we have the implication $0<\delta \le d(x,D^c) \implies   d(x,D^\delta) \le \delta \le R_D $.
It follows that
$d(x,D^c) \le \rho(x) \le R_D$. Moreover, $\rho(x) = R_D$ for every $x\in U_D=\Omega_{\frac {1}{R_D}}$.
We now prove \eqref{equi2}. Let us recall that $\O_\l$ given by \eqref{def:Olambda} coincides with an enlargement of $D^\delta$, namely 
$$\O_\l= \{x\in D \, :\, d(x, D^\delta)<\delta\} \quad \text{where $\delta = \l^{-1}$} .$$ 
Therefore, by the definition of $\rho$ in \eqref{def:rho},  we have  $\rho(x) >\l^{-1} $ for any $x\in \O_\l$. Conversely, assume that $\rho(x) >\l^{-1} $. Then, as we know from 
Lemma \ref{alpha}, it holds  $d(x, D^\delta)<\delta$ for any $\delta \in [0,\rho(x)]$, hence in particular for $\delta=\l^{-1}$, whence $x\in \O_\l$.

\end{proof}

\begin{remark}\label{rhoAlp}
By Lemma \ref{lemrho1}, it follows that the level sets of $\rho$  given by $C^\delta:=\{x\in D \sepa \rho(x) =\delta\}=\partial\Omega_{\frac{1}{\delta}}\cap D$ will give a partition to $D$, namely
\begin{align*}
D = \bigcup_{0<\delta \le R_D} C^\delta .
\end{align*}
Note that, for $\delta \in (0,R_D)$,  the sets  $C^\delta$  are arcs of radius $\delta$, while for $\delta=R_D$, we obtain 
 $C^{R_D} = U_D =\O_{\frac1{R_D}}$ which is a convex open subset of $D$.  
\end{remark}

\begin{proposition} \label{Lip-rho} The cut-locus potential $\rho$ is continuous on $\ov D$ and locally Lipschitz in $D$  (its gradient blows-up as $d(x, D^c)\to 0$). Moreover its trace on $\partial D$ satisfies
\begin{align}\label{rho-partialD}
\rho(x) \ =\ \tau(x)\ \le \frac1{\kappa_{\partial D}(x)} \quad \text{for all $x\in \partial D$}.
\end{align}
\end{proposition}

\begin{proof}
Firstly, we prove that $\rho$ is locally Lipschitz (hence continuous) in $D$.
Given $\delta>0$, for every $x\in \Omega_{\delta^{-1}}\cap D$, we have
\begin{align*}
d(x,D^c) >0  \qquad \tinm{and} \qquad \delta = |x-\proj_\delta (x)|.
\end{align*}
We set $r:=d(x,D^c)$. Let $z$ be the point lying outside the disk $B(\proj_\delta (x),\delta)$, on the line passing $x$, $\proj_\delta (x)$ such that $|z-x|= r$. So, $z$ is in $\ov D$. For every $\delta' <\delta$, we take $y$ as the point inside $D$, on the same latter line such that the disks $B(y,\delta')$ and $B(\proj_\delta (x),\delta)$ have the same tangents passing $z$. See Figure \ref{fig:rhoL} for our settings.

\begin{figure}[H]
\centering
\includegraphics[scale=1.2]{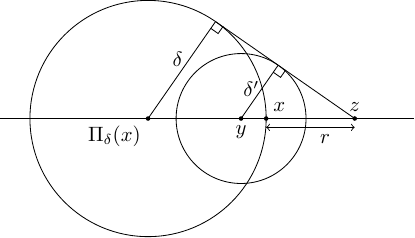}
\caption{To prove that $\rho$ is locally Lipschitzian.}
\label{fig:rhoL}
\end{figure}

\noindent
Thales' Theorem is applied,
\begin{align*}
\frac{|z-y|}{|z- \proj_\delta (x)|}=\frac{\delta'}{\delta},
\end{align*}
then, we get 
\begin{align}
|z-y|=\frac{\delta'}{\delta} (\delta+r). \label{eq1}
\end{align}
We notice that $B(\proj_\delta (x),\delta) \subset \ov D$ and, $z\in \ov{B(x,r)} \subset \ov D$. Since $D$ convex,   we have 
$$\conv\left[\ov{B(\proj_\delta (x),\delta)}\cup\{z\} \right] \subset \ov D.$$
This implies $B(y,\delta')\subset \ov D$. Thus, $ y\in \ov{D^{\delta'}}$ and, by using the equality \eqref{eq1}, it holds
\begin{align*}
d(x,D^{\delta'})\le |x-y|=|z-y|-r=\delta'-r(1-\frac{\delta'}{\delta}).
\end{align*}
We now can summarize that 
\begin{align}
\forall x\in \partial\Omega_{\delta^{-1}}\cap D,\; \forall \delta'<\delta,\quad
d(x,D^{\delta'})\le \delta' - d(x, D^c)(1-\frac{\delta'}{\delta}).\label{state1}
\end{align}
We observe that for each $x\in D \setminus \ov{\Omega_{\frac 1R}}$, by using Remark \ref{rhoAlp}, $x\in \partial\Omega_{\delta^{-1}}\cap D$ with $\delta=\rho(x)$.
As a consequence of  \eqref{state1}, for every $x,x'\in D \setminus \ov{\Omega_{\frac 1R}}$ and, for every
$\delta' <\delta =\rho(x)$, we have
\begin{align*}
d(x',D^{\delta'})-\delta' &\le |x-x'| + d(x,D^{\delta'}) - \delta' \\
&\le |x-x'| - d(x, D^c) \left[ 1-\frac{\delta'}{\delta} \right].
\end{align*}
If $\rho(x') <\rho(x)$, we can choose $\delta'=\rho(x')$ and, then $d(x',D^{\delta'})-\delta'=0$. We infer that
\begin{align*}
d(x,D^c)\left( 1-\frac{\rho(x')}{\rho(x)} \right)\le |x-x'|.
\end{align*}
Therefore, for every $x,x'\in D\setminus \ov {\Omega_{\frac 1R}}$ such that $\rho(x') <\rho(x)$, we have
\begin{align*}
|\rho(x')-\rho(x)| \le \frac{\rho(x)}{d(x,D^c)}|x'-x|.
\end{align*}
Finally, we conclude that, if $x\in D\setminus \ov {\Omega_{\frac 1R}}$ and $d(x,D^c)<\e$, we have the following inequality holding for every $x'\in B(x,\e)$:
\begin{align}
|\rho(x')-\rho(x)|\ \le\ \frac{R}{d(B(x,\e),D^c)}\, |x'-x|.
\label{rhoLip}
\end{align}
Recalling that $\rho$ is constant on $\ov{\Omega_{\frac 1R}}$, we conclude that $\rho$ is locally Lipschitz  in $D$.

\medskip
In a second step,  we show \eqref{rho-partialD}. The fact that $\rho(x)=\tau(x)$ for every  $x\in \partial D$ is a direct consequence of the equality
$\gamma(x)=\tau(x)$ (see \eqref{gamma=tau}) and of the fact that $\{\delta \ :\  \alpha(x,\delta) \le 0 \} = [0,\gamma(x)]$.
Note that, if $x\in \partial D$, then $ \alpha(x,\delta)= 0$ whenever $\delta\le \tau(x)$ and we know from Lemma \eqref{alpha} that $\alpha(x,\cdot)$ is stricly increasing on $[\gamma(x), R]$.

 Next we  show  that $\tau(x) \le  \frac{1}{\kappa_{\partial D}(x)} $  at any $x\in \partial D$. Without loss of generality, we can assume that
 $\tau(x)>0$.  Then, by the definition of $\tau(x)$, this means that $x\notin \Lambda_D$, hence $x\in \partial_r D$ and $d(x, D^\delta) = \delta$ whenever $\delta\le \tau(x)$. In particular the equality $d(x, D^{\tau(x)})= \tau(x)$ means that the ball $B(x -\tau(x) \nu_D(x)$ is contained in $D$ and touches the boundary $\partial D$ at $x$. It follows that $\kappa_{\partial D}(x)<+\infty$ and that  $0< \rho(x)=\tau(x) \le \frac{1}{\kappa_{\partial D}(x)}$. 
 
\medskip
 In a last step, we show that $\rho$ is continuous in all $\ov D$. First we notice that $\rho$ is lower semicontinuous since any strict upper level set 
 $\{\rho > r\}$ coincides with the open subset $\Omega_{\frac1{r}}$ for every $r\in (0,R_D)$ , while it coincides with $D$ for $r=0$ and the empty set for $r\ge R_D$.
 
 \medskip
 Let $x_n\in \ov D$ such that $x_n\to x$. The case where $x\in \ov U_D$ is easy since we have $\rho(x_n)\le \max \rho= R_D= \rho(x)$ and
 by the lower semicontinuity of $\rho$, $ \rho(x) \le \liminf \rho(x_n) \le \limsup \rho(x_n) \le \rho(x)$.
 Therfore we can  assume that $x\notin \ov U_D$.   Then  $x_n \notin \ov U_D$ for large $n$ and,  
 in virtue to the characterization \eqref{master-rho}, we have 
$$\rho(x_n) =d(x_n, D^{\rho(x_n)}) $$
 As $\rho(x_n) \le R$, we can assume that, up to extracting a subsequence,  $\rho(x_n) \to \delta^*$  for some $\delta^*$. 
  By the continuity property given in the assertion (i) of  Lemma \ref{stmonotone}, passing to the limit in the equality above leads to the equality
   $\delta^* =d(x, D^{\delta^*})$. Then, by Lemma \ref{alpha}, we have either $\delta^*= \rho(x)$ or $\delta^*=0$. 
   If $\delta^*>0$, we are done since, in this case, the whole sequence $\rho(x_n)$ converges to the unique cluster point $\rho(x)$.
   If $\delta^*=0$, then $\rho(x_n)\to 0$ and, thanks the lower semicontinuity of $\rho$, we infer that $\rho(x)=0$. Since $\rho(x) \ge d(x,D^c)$ while
    $\rho(x)=\tau(x)$ on $\partial D$, this is possible only if $x\in \partial_s D$ where $\tau(x)=0$. However, even in this case, we have $\rho(x_n) \to \rho(x)$.
  
 \end{proof}

%

 \begin{remark}
Y. Li and L. Nirenberg proved in \cite{li2005distance} that $\tau$ is Lipschitz if $D$ has a $C^{2,1}$ boundary but 
it is untrue for a general convex domain  (even  $C^{2,\alpha}$ with $\alpha<1$ is not enough).
We conjecture that for a general convex domain $D\subset \R^2$, the cut-locus potential $\rho$ belongs to $C^{0,\frac1{2}}(D)$, as it is the case for a square (see  the example \ref{ExRho} below).
\end{remark}

\begin{example}[An explicit formula for $\rho$ in the case of a square] \label{ExRho}~

Let us consider the domain  $D=(-1/2,1/2)^2$ whose inradius is $R=\frac1{2}$. Hence   $\rho= 1/2$ on the disk $U_D=\Omega_2= \{|x|\le \frac1{2}\}$. 
The cut-locus $\ov{\Lambda}$ consists of the two diagonals of $D$ whereas the Cheeger constant is  $h_D=2+\sqrt{\pi}$ . For every $x\in D\setminus\Omega_2$, $\delta=\rho(x)$ is the unique  $\delta$ such that $x$ belongs to the arc of circle  $C^\delta= \partial \O_\l\cap D$.
Let us compute $\rho$ in the north east quater $[0, 1/2]^2\setminus\ov \Omega_2$. We set $x=(x_1,x_2)$, and $\delta=1/2-t$ for $t\in (0,1/2)$. Then, $x\in C^\delta$ (hence $\rho(x) =\frac1{2}-t$)  if and only if $t$ solves the equation
\begin{align*}
\begin{cases}
(x_1-t)^2 + (x_2-t)^2 = (\frac{1}{2} -t )^2\\
x_1^2 + x_2^2 \ge \frac{1}{4}.
\end{cases}
\end{align*}
These equations determine a unique $t\in [0,\frac1{2}]$  given by
\begin{align*}
t = x_1 +x_2 -\frac{1}{2} - \sqrt{2}\sqrt{\left(\frac{1}{2} - x_1 \right) \left( \frac{1}{2}- x_2 \right)}.
\end{align*}
Accordingly we obtain the following expression for $\rho$ for $x\in D$: 
 \begin{align*}
\rho(x)= \begin{cases}  1 - (|x_1| + |x_2|) + \sqrt{2}\sqrt{\left(\frac{1}{2} - |x_1| \right) \left( \frac{1}{2}- |x_2| \right)} &  \text{if}\  x_1^2+ x_2^2\ge 1/4 \\
\frac1{2}  & \text{otherwise}
\end{cases}
\end{align*}
This computation confirms that $\rho$ vanishes only at the vertices of $\ov D$ where the curvature is infinite. It is positive  and  of class $C^1$  inside $D$. After some computations, we get the following equality:
\begin{align*}
|\nabla\rho(x)| = \frac{\rho(x)}{\sqrt{2}\sqrt{\left(\frac{1}{2} - |x_1| \right) \left( \frac{1}{2}- |x_2| \right)}}
\quad \text{on}\  \{x_1^2+ x_2^2\ge 1/4\} ,
\end{align*}
which clearly shows that $|\nabla \rho|$ blows-up when approaching  the sides of the square.  The presence of the square term in the expression of $\rho(x)$ indicates that we cannot expect better than the $C^{1/2}$ regularity of $\rho$ in $\ov D$. 
We present  some calculations for the normalized gradient of $\rho$ on $D\setminus \ov U_D.$ 
\begin{align*}
\nabla \rho(x) = \left(
\begin{array}{c}
-1 - \frac{\sqrt{2}}{2} a(x)  \\
-1 - \frac{\sqrt{2}}{2} \frac{1}{a(x)}
\end{array} \right)\  \text{where}
\quad  a(x) := \frac{ \sqrt{\frac{1}{2}- x_2 }}{\sqrt{\frac{1}{2} - x_1 }}\ ,
\end{align*}
 The unit vector $q_\rho:=  - \frac{\nabla \rho}{|\nabla\rho|}$ is given by:
\begin{align*} &q_\rho(x)  = \frac{1}{\sqrt{C_1(x)^2 + C_2(x)^2}} \left( \text{sgn}(x_1)\, C_1(x),  \text{sgn}(x_2)\, C_2(x) \right) ,\\
&C_{1}(x) := 1 + \frac{\sqrt{2}}{2} a(x)  \quad \text{and} \quad C_{2}(x) :=  1 + \frac{\sqrt{2}}{2} \frac1{a(x)} .
\end{align*}
%
From the expressions above, we can check that, if $x$ approaches the vertical sides (i.e.$|x_1| \to 1/2$), then $a(x) \to +\infty$ so that 
 $q_\rho(x)\to ({\rm sgn}(x_1), 0)$ ; on the same way  $q_\rho(x)\to (0, {\rm sgn}(x_2))$ as $|x_2|\to 1/2$. Therefore the normal trace of $q_\rho$ on $\partial D$
 satisfies the equality  $q_\rho  \cdot \nu_D=1$. 
 This property will be confirmed in the general case in the forthcoming Theorem \ref{thmcutl}.
 A representation  of  $\rho$ and of its normalized gradient on the first quarter $[0,1/2]^2$ is displayed in Figure \ref{fig:rhoDs}.  In the left subfigure, the level lines of $\rho$ are drawn in varied colors while, in the right one,  the normalized gradient of $\rho$ is represented (with a magnifying glass)  in black streamlines  starting 
 from the circle $|x|=1/2$ .
\end{example}

\begin{figure}[H]
\centering
\includegraphics[scale=0.278]{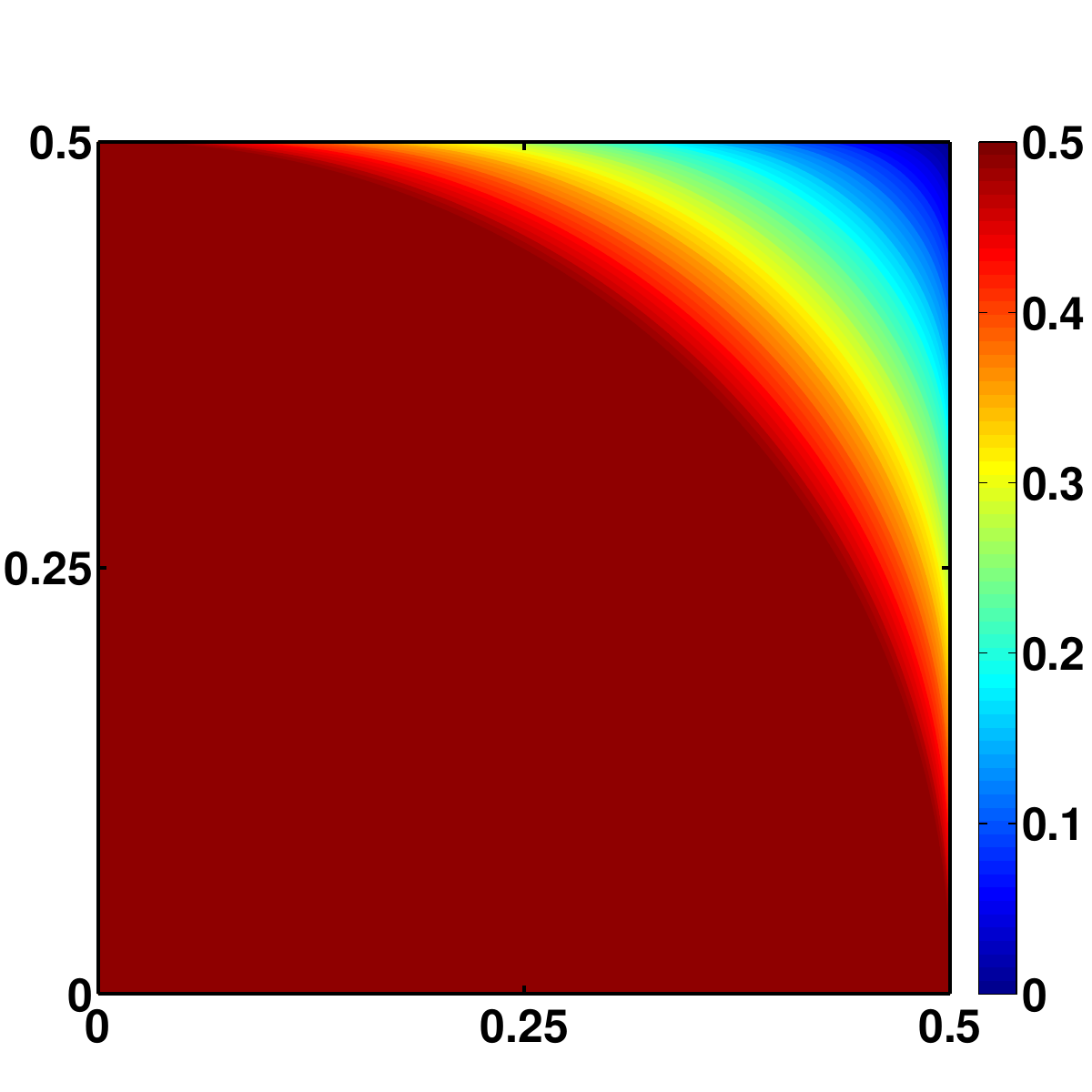} \;
\begin{tikzpicture}[spy using outlines={circle,
  magnification=8, size=4.4cm, connect spies}]
\node[inner sep=0pt] at (0,0){
\includegraphics[scale=0.256]{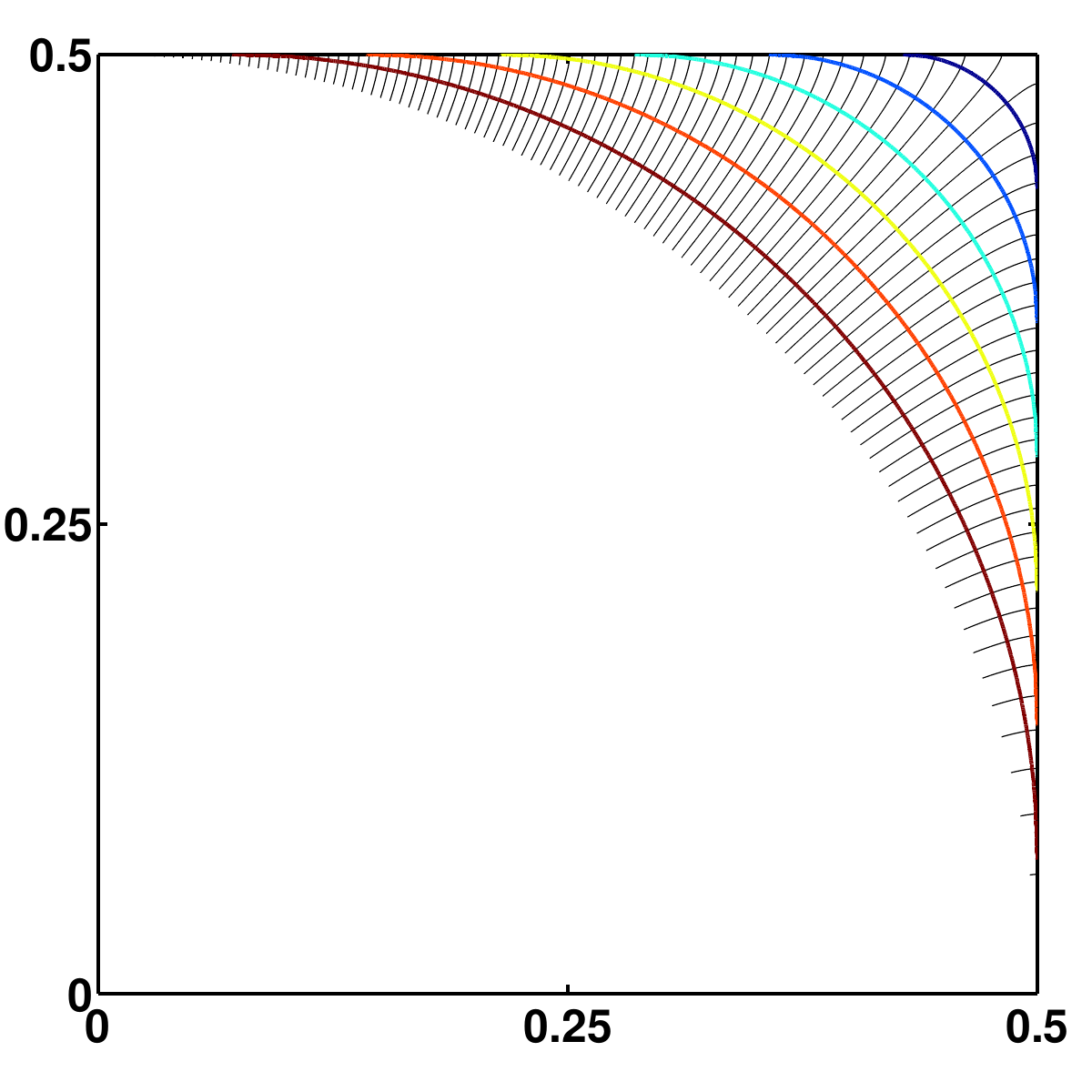}};
\spy [black] on (2.,0.5) in node at (5.,0.);
\end{tikzpicture}
\caption{ $\rho$ and the stream lines of  $q_\rho= -\frac{\nabla \rho}{ |\nabla\rho|}$ when $D=(-\frac1{2},\frac1{2})^2$.} \label{fig:rhoDs}
\end{figure}

\begin{theorem}\label{thmcutl}
Let $D\subset \RR^2$ be a bounded convex domain of inner radius $R$ and central set $U_D$.  Then, on the open subset $D\setminus U_D= \{0<\rho <R\}$, the locally Lipschitz potential $\rho$ is such that $\nabla \rho \not= 0 \ \text{ a.e.}$. Moreover the unit vector field $q_\rho:=- \frac{\nabla\rho}{|\nabla\rho|}$
satisfies
 \begin{equation}\label{pde-rho}
  \Div q_\rho = \frac1{\rho} \quad \text{in $D\setminus \ov U_D$}\quad \text{and} \quad  q_\rho\cdot\nu_D= 1  \quad \HH^1 \ \text{a.e. on $\partial D\setminus \ov U_D$}.
\end{equation}
Here above, the first equality is intended a.e. and in the sense of distributions in $D\setminus \ov U_D$. 
\end{theorem}

\begin{proof}
Note that the open subset $\{0<\rho <R\}$ is indeed $D\setminus \ov{\Omega_{\frac 1R}}$ where $\rho$ is not constant. On this subset, the vector field $\eta:=-\nabla \rho/|\nabla \rho|$ can be rewritten as
\begin{align*}
\eta(x)= \frac{x- \proj_{\rho(x)} (x)}{\rho(x)}.
\end{align*}
Owing to the expression above, it is clear that the equality $\eta(x)=\nu_D(x)$ holds at any regular point of the boundary, that is  for every $x\in \partial_r D\setminus \ov U_D$.
Indeed,  at those points, we have $\rho(x)=\tau(x)$ and $x-\proj_{\rho(x)} (x)= \rho(x) \nu_D(x)$. 

Next we are going to prove that $\proj_{\rho(x)} (x)$ is locally Lipschitz in $x$ and hence, so is $\eta$.
We firstly claim that for every $\delta, \delta'$ satisfying $0<\delta'<\delta$, for each $y\in D$, there is a constant $K_\delta$ such that 
\begin{align}
|\proj_\delta (y)- \proj_{\delta'} (y)| \le K_\delta |\delta-\delta'|.
\label{excessb}
\end{align}
It follows immediately that for every $x\in D\setminus \ov{\Omega_{\frac 1R}}$, $x'\in B(x,\e)\subset D$, keeping in mind \eqref{rhoLip},
\begin{align*}
|\proj_{\rho(x)} (x)- \proj_{\rho(x')} (x')|
&\le | \proj_{\rho(x)} (x)- \proj_{\rho(x')} (x)| + |x-x'|\\
&\le K_{\rho(x)}|\rho(x)-\rho(x')|+|x-x'| \\
&\le \left( K_{\rho(x)} + \frac{R}{d(B(x,\e),D^c)} \right)|x-x'|.
\end{align*}
Hence, we obtain that $\eta$ is locally Lipschitz on $D$ provided 
we show the validity of the claim \eqref{excessb}. 
Given $0<\delta' <\delta$, for every $y\in D$, by Lemma \ref{lem:mod1} (ii), $\proj_{\delta'} (y)$ is always in  $M^\delta_{\delta'}( \proj_\delta (y) )$, see Figure \ref{fig:proof} for illustration. We obtain
\begin{align}
|\proj_\delta (y) - \proj_{\delta'} (y)| \le |w-x|=\frac{|\delta-\delta'|}{\cos\varphi(x)},
\label{ineq3}
\end{align}
where $x= \proj_\delta (y)$ and $w$ is the extreme point of $M^\delta_{\delta'}(x)$ in the complement of $D^{\delta'}$ (see Figure \ref{fig:proof}).
By exploiting the assertion (iii) of Lemma \ref{access} , there exists $K_{\partial D^\delta}>0$ such that 
\begin{align*}
K_{\partial D^\delta}
=\min \left\{ k_{\partial D}(s) \sepa s\in \partial D^\delta  \right\} = \min \left\{ \cos^2\varphi(s) \sepa s\in \partial D^\delta  \right\}.
\end{align*}
Since $x\in \partial D^\delta$, we have
\begin{align}
\frac{1}{\cos\varphi(x)} \le \frac{1}{\sqrt{K_{\partial D^\delta}}}.
\label{ineq4}
\end{align}
We then use the inequalities \eqref{ineq3} and \eqref{ineq4} to derive that the inequality \eqref{excessb} holds  with $K_\delta = (K_{\partial D^\delta})^{-1/2}$

\begin{figure}[H]
\centering
\includegraphics[scale=1]{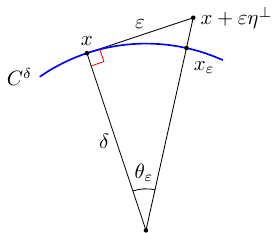}
\caption{Divergence of $\eta$ along $C^\delta$.}\label{fig:diveta}
\end{figure}

The next step is to prove that $\displaystyle \Div \eta = \frac{1}{\rho}$ in $D\setminus \ov{\Omega_{\frac 1R} }$. The vector field $\eta$ is indeed the unit normal to the level sets $C^\delta= \{x\in D \sepa \rho(x) = \delta\}$ which are arcs of radius $\delta$. Given $x$ and $x_\e$ on $C^\delta$ illustrated by Figure \ref{fig:diveta}. Let us evaluate locally divergence of $\eta$ along directions $\eta$ and $\eta^\perp$.
We recall that
\begin{align*}
(D\eta)h\cdot h:= \lim_{\e\to 0} \left\langle \frac{\eta(x+\e h)-\eta(x)}{\e},h \right\rangle
\end{align*}
for some non null direction $h$.
Since $|\eta|=1$, we have
\begin{align*}
2(D\eta)h\cdot \eta =
\left\langle \frac{\eta(x+\e h)-\eta(x)}{\e}, \eta(x+\e h) +\eta(x) \right\rangle = \frac{1}{\e} \left[ |\eta|^2(x+\e h) -|\eta|^2(x) \right] = 0.
\end{align*}
Thus, for $h=\eta$, we get 
\begin{align}
(D\eta)\eta\cdot \eta =0. \label{eq2}
\end{align}
As $\eta$ locally Lipschitz, there is some constant $M$ such that
\begin{align*}
\left| \frac{\eta(x+\e h)- \eta(x_\e)}{\e} \right| \le \frac{M}{\e}|x+\e h -x_\e| = \frac{M\delta}{\e}\left(\frac{1}{\cos\theta_\e} -1 \right) = \frac{M\delta}{\e}\left(\sqrt{1+ \frac{\e^2}{\delta^2}} -1 \right) \sim \frac{M\e}{2\delta}.
\end{align*}
Hence, for $h=\eta^\perp$,
\begin{align*}
\left\langle \frac{\eta(x+\e h)-\eta(x)}{\e},h \right\rangle 
&= \left\langle \frac{\eta(x+\e h)-\eta(x_\e)}{\e},h \right\rangle + \left\langle \frac{\eta(x_\e)-\eta(x)}{\e},h \right\rangle \\
& \sim \frac{M\e}{2\delta} + \frac{\sin\theta_\e}{\e} 
\sim \frac{M\e}{2\delta} + \frac{\tan\theta_\e}{\e} \\
&= \frac{M\e}{2\delta} + \frac{1}{\delta}.
\end{align*}
We obtain 
\begin{align}
(D\eta)\eta^\perp\cdot \eta^\perp = \frac{1}{\delta}.
\label{eq3}
\end{align}
From the equations \eqref{eq2} and \eqref{eq3}, we can derive that 
\begin{align*}
\Div \eta = (D\eta)^T : Id =(D\eta)^T : (\eta\otimes \eta + \eta^\perp\otimes \eta^\perp)= (D\eta)\eta\cdot \eta+ (D\eta)\eta^\perp\cdot \eta^\perp = \frac{1}{\delta} = \frac{1}{\rho}.
\end{align*}
As $\eta$ is locally Lipschitz on the open set $\{0<\rho< R\}$, the equality above holds not only a.e. on this set but also in the distributional sense.
\end{proof}

\begin{remark} \label{newpb} In view of Theorem \ref{thmcutl}, $\rho$ solves the boundary value problem 
$$  \Div \frac{\nabla \rho}{|\nabla\rho|}  + \frac1{\rho}=0 \quad \text{in $D\setminus \ov U_D$}\ , 
\quad  \frac{\nabla \rho}{|\nabla\rho|}\cdot \nu_D =- 1 \quad \text{in $\partial D\setminus \ov U_D$}\ ,\quad \rho=R_D\quad \text{in $\partial U_D $ .}$$
In fact it can be shown that  $\rho$ is directly  related to the following strictly convex variational problem
$$  \min \aco{ \int_D |\nabla u| + \int_{\partial D} |u|\, d\HH^1 - \int_D \ln(u) \, dx \sepa   u\in W^{1,1}(D) } ,$$
whose unique solution $\ov u$ admits a maximal plateau $\ov u= \frac1{h_D}$ and coincides with $\rho$ on the complement of the Cheeger set of $D$.
Studying this new  minimization problem in dimension $N \ge 2$ could help us  understand the structure of solutions  to $m(\l,D)$ in the case of a general domain $D\subset \R^N$, especially when the convexity assumption of the domain $D$ is removed.
\end{remark}

%
%
%


\subsection{ The set $\Omega_\lambda$ is calibrable and solves $m(\l,D)$}
\label{mlOl}
By exploiting the PDE \eqref{pde-rho} satisfied by the cut-locus potential of $D$,   
we are now in position to derive, as corollaries of Theorem \ref{thmcutl}, the calibrability and the optimality of the set $\O_\l$ for $m(\l,D)$.

\begin{corollary}\label{lemcalibr2}
Let $\l \ge h_D$.  Then $\O_\l$ defined by \eqref{def:Olambda} solves $m(\l, \O_\l)$. Therefore $\l\ge \l_{\O_\l}$ and $\O_\l$  is  $\theta$-calibrable with constant $\theta\le \lambda\, \lambda_{\O_\l}^{-1}$.
\end{corollary}


\begin{proof} By \eqref{minlambda}, it is enough to show that $\O_\l$ solves $m(\l, \O_\l)$. Let us apply  the optimality conditions of the assertion (ii) of Theorem \ref{duality2} in the case where  $D=\O= \O_\l$. Then we are done if we can find $\ov q\in L^\infty(\O_\l;\R^2)$ such that
\begin{align}
 |\ov q| \le 1 \quad\tinm{a.e. in} \O_\l\ ,\quad 0 \le \Div \ov q \le \lambda \quad\tinm{a.e. in} \O_\l \ ,\quad
 \ov q\cdot \nu_{\O_\l} =1 \quad\tinm{$\HH^1$-a.e. on} \partial \O_\l\label{doptOl}
\end{align}
Such a vector field can be  constructed on $\O_\l$ by starting from a calibrating field of the Cheeger set of $D$. This semi-explicit construction is provided with the existence of a calibrating field for the Cheeger set of $D$, that means $\ov q= q_{h_D}$ in $\Omega_{h_D}$ where $q_{h_D}\in L^\infty(\Omega_{h_D};\RR^2)$ satisfies
\begin{align*}
|q_{h_D}| \le 1,\quad \Div q_{h_D} = h_D \quad\tinm{a.e. in} \Omega_{h_D},\qquad q_{h_D} \cdot \nu_{\Omega_{h_D}} =1 \quad\tinm{$\HH^1$-a.e. on} \partial \Omega_{h_D}.
\end{align*}
We need  now  to construct $\ov q$ on $\O_\l\setminus \ov{\O_{h_D}}$.  The cut-locus potential $\rho$ comes into play here.
Let us define $\ov q$ on $\O_\l$ as follows: 
\begin{align}\label{calib-Olambda}
\ov q (x) := \begin{cases}  q_{h_D}(x) & \ \text{if $x\in \O_{h_D}$} \\
- \frac{\nabla \rho(x)}{|\nabla \rho|(x)} & \ \text{if $x\in \O_\l\setminus \ov{\O_{h_D}} $} \end{cases}
\end{align}
The condition $ |\ov q| \le 1$ is clearly satisfied. On the other hand, recalling that $\O_{h_D}=\{\rho>h_D^{-1}\}$,  $\rho$ is constant  on the interface $\partial \O_{h_D} \cap D$ while $h_D\le \rho^{-1}\le \l$ in $\O_\l\setminus \O_{h_D}$.  Therefore the normal trace of $\ov q$ has no jump and  the ditributional divergence  $\Div \ov q $ belongs to $L^\infty(\O_\l;  [h_D, \l])$ since
$$  \Div \ov q = \begin{cases} h_D & \ \text{if $x\in \O_{h_D}$} \\
\frac1{\rho} &\ \text{if $x\in \O_\l\setminus \O_{h_D}$} \end{cases}.$$
Eventually, the condition $\ov q\cdot \nu_{\O_\l} =1$ is fullfiled on $\partial\O_\l \cap D$ where $\rho =\frac1{\l}$. 
The same equality holds on the shared boundary piece $\partial\O_\l \cap \partial D$ since, by \eqref{pde-rho}, it holds  $\eta=\nu_D$ on $\partial D \setminus \ov U_D$. 
Finally, the three  conditions in \eqref{doptOl} are satisfied.
\end{proof}

 \begin{corollary}\label{opti-mlD}
Let $ \lambda > h_D$. Then,  the upper level set 
$\O_\l=  \{ \rho > \lambda^{-1}\} $
is the unique solution to the problem $m(\lambda,D)$.
\end{corollary}

We remark that, according to Theorem \ref{duality2}, the optimality of $\Omega_\l$ for the problem $m(\l,D)$ can be achieved by extending the construction
 of the  calibrating field $\ov q$ defined in  \eqref{calib-Olambda} from  $x\in \Omega_\l$ to $x\in D\setminus \O_\l$. This will be done in the next subsection.  
Here we propose a  direct  proof based on the PDE satisfied by the cut-locus potential of $D$, as a consequence of Theorem \ref{thmcutl}. 
Note that a similar PDE proof appears in  \cite[Proposition 4]{alter2005characterization}  exploiting the solutions of a family of auxiliary problems where a Neumann boundary condition is implicitely imposed on $\partial D$.

\begin{proof}
Let  consider  $F \subset D$ a Borel subset with finite perimeter.  We need need to show that $P(F) -\l |F| \ge P(\O_\l)- \l |\O_\l|$.
Let $\ov q$ be the vector field satisfying \eqref{doptOL2} that we constructed in the proof of Corollary \ref{lemcalibr2}. 
Then, setting $\delta= \frac{1}{\l}$, we have the equalities:
$$P(\O_\l)- \l\, |\O_\l|= \int_{\O_\l} (\Div \ov q -\l)\, dx = \int_{\O_\l} \left( \frac1{\rho} - \frac1{\delta}\right)\, dx.$$
 On the other hand, since $|\ov q|\le 1$, we have:
\begin{align*}  P(F)- \l\, |F| &\ge  \int_F (\Div \ov q -\l)\, dx = \int_F \left( \frac1{\rho} - \frac1{\delta}\right)\, dx \\
& =  \int_{\O_\l} \left( \frac1{\rho} - \frac1{\delta}\right)\, dx +  \int_{F\setminus \O_\l} \left( \frac1{\rho} - \frac1{\delta}\right)\, dx -  \int_{\O_\l\setminus F} \left( \frac1{\rho} - \frac1{\delta}\right)\, dx
\end{align*} 
Since $\frac1{\rho} \ge \frac1{\delta} \ge 0$ on $F\setminus \O_\l$ and $\frac1{\rho} \ge \frac1{\delta} \le 0$ on $\O_\l\setminus F$, we deduce that
$$P(F)- \l\, |F| \ge \int_{\O_\l} \left( \frac1{\rho} - \frac1{\delta}\right)\, dx = P(\O_\l) - \l |\O_\l| ,$$
whence the optimality of $\O_\l$ for $m(\l,D)$ for every $\l\ge h_D$.

Let us prove now the uniqueness of the solution of $m(\l,D)$ for  $\l>h_D$. From the previous arguments, we know now that for all $\mu,\nu$ such that $h_D<\mu<\l<\nu$, the sets $\O_\mu$ and $\O_\nu$ are minimal for $m(\mu,D)$ and $m(\nu,D)$ respectively. Then, by a well known comparison argument (see for instance \cite[Lemma 4 (i)]{alter2005characterization}), we have the inclusions $\O_\mu \subset \O \subset \O_\nu$  for any $\O$ solving $m(\l,D)$.  
The equality $\O=\O_\l$ follows by sending $\mu \nearrow \l$ and $\nu \searrow \l$.

\end{proof}

\subsection{Extending the calibration field to all $D$}
\label{sec:{calib-D}
} 

We will now extend the vector field $\ov q$ prealably defined in subsection \ref{mlOl} (see \eqref{calib-Olambda}) from $\O_\l$ to $D$.
To that aim we are going to  design a unit vector field $q_\lambda$ in $D\setminus \ov{\Omega_\lambda}$ such that 
\begin{align}  \label{qlambda}
|q_\l| \le 1 \; \tinm{in} D\setminus \ov{\Omega_\l}, \qquad
\Div q_\l = \l  \; \tinm{in} D\setminus \ov{\Omega_\l}, \qquad q_\l \cdot \nu_{\O_\l} = 1 \;\tinm{on} \partial \Omega_\lambda\cap D.
\end{align}
The subset $D\setminus \ov{\Omega_\lambda}$  involved in the forthcoming construction is represented in Figure \ref{Fig:omegaLambda} below in the case of a square 
\begin{figure}[H]
\centering
\includegraphics[scale=1]{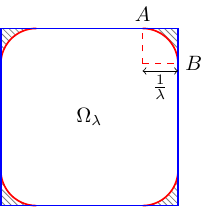}
\hspace{2cm}
\includegraphics[scale=1]{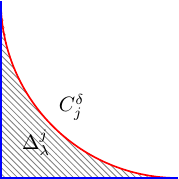}
\caption{$D\setminus \ov{\O_\l}$ represented in dash}\label{Fig:omegaLambda}
\end{figure}

Then, by assigning to $\ov q$ the value $\ov q=q_\l$ on $D\setminus \O\l$, one check easily that the resulting $\ov q\in L^\infty(D;\R^2)$ 
satisfies  the optimality conditions given in the assertion (ii) of Theorem \ref{duality2}  when $\O=\O_\l$, namely:
 \begin{align}
& |\ov q| \le 1 \quad\tinm{a.e. in} D, && 0 \le \Div \ov q \le \lambda \quad\tinm{a.e. in} D,\label{doptOL1}\\
& \ov q\cdot \nu_{\Omega_\lambda} =1 \quad\tinm{$\HH^1$-a.e. on} \partial \Omega_\lambda,  && \Div \ov q =\lambda \quad\tinm{a.e. in} D\setminus \Omega_\lambda .\label{doptOL2}
\end{align}
Accordingly, we will obtain another proof for the optimality of $\O_\l$.
The global construction of $\ov q$ concerns the three subsets  represented in Figure \ref{finclusion} (see Figure \ref{Fig:EomegaLambda} in the case of an ellipsoid).
For a square,  infinite curvature occurs at the corners, so that $\kappa_\infty(\partial D) = +\infty$ and the inclusion $\O_\l \subset D$ is strict for every $\l\ge h_D$.

\begin{figure}[H]
\centering
\qquad\quad
\includegraphics[scale=0.8]{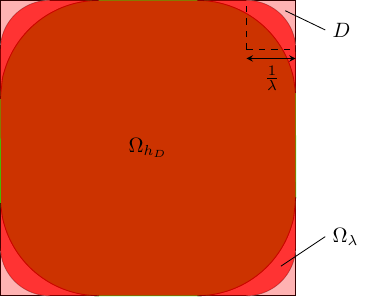}
\caption{The inclusion $\Omega_{h_D}\subset \Omega_\lambda\subset D$.}\label{finclusion}
\end{figure}

\bigskip
We end this subsection and conclude the paper by giving:

\begin{enumerate}
\item [-] the detailed construction of the field $q_\l$ satisfying \eqref{qlambda} (see Lemma \ref{last}).
\item [-] the closed form of $\ov q$ in case of a square and of an ellipsoid (see Example \ref{qexpli}) .

\end{enumerate}

\begin{lemma}\label{last}
There exists a vector field $q_\lambda$ in $D\setminus\ov{\Omega_\lambda}$ satisfying \eqref{qlambda}

\end{lemma}

\begin{proof}
From now  on, we set $\delta=1/\lambda$ and  adapt the notation used at the beginning of Section \ref{scutlocus}  to
the convex subset  $ D^\delta$. Recall that we denote by $\Pi_\delta(x)$ the unique projection of $x$ on $\ov{D^\delta}$. The singular part of its boundary $\partial_s D^\delta$ has  at most countably many points, and we set
$$ \partial_s{D^\delta} = \{ x_j : j\in J\}\quad \text{ (being $J$  either empty, or finite, or countable)} .$$
 For each point $x\in \partial_s D^\delta$, the normal cone of $D^\delta$ at the point, denoted by $N_{D^\delta}(x)$, is generated by the two limit vectors $\nu_{D^\delta}^-(x)$ and $\nu_{D^\delta}^+(x)$. $\varphi^\delta_-(x)$ and $\varphi^\delta_+(x)$ are the corresponding angles of the two vectors in $S^1$. Accordingly we define the sets $N^\delta(x)$, $C^\delta(x)$, and $M_{\delta'}^\delta(x)$ as follows:
\begin{gather*}
\quad N^\delta(x):= x + \aco{ p\in N_{D^\delta}(x) \sepa |p|\le \delta }, \tinm{for} x \in \ov{D^\delta}, \\
 \quad C^\delta(x):= x+\aco{p\in N_{D^\delta}(x) \sepa  |p|=\delta }, \tinm{for} x\in \partial D^\delta,\\
M_{\delta'}^\delta(x):= x+ \aco{ p\in N_{D^\delta}(x) \sepa \left\langle p, \nu_{D^\delta}^\pm (x) \right\rangle \le \delta-\delta' }, \tinm{for} x\in \ov{D^\delta},\; 0\le \delta' <\delta <R. 
\end{gather*}
\begin{figure}[H]
\centering
\includegraphics[scale=0.6]{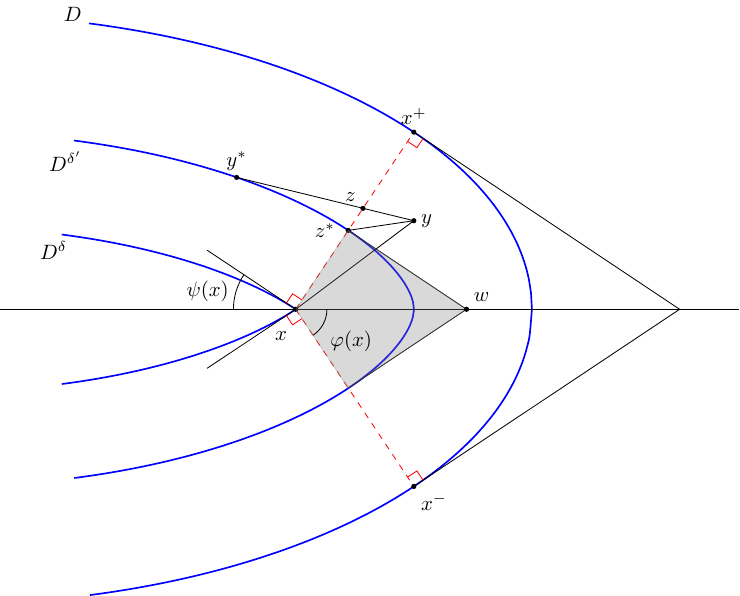}
\caption{ The kite in gray color represents the set $M_{\delta'}^\delta$(x).}\label{fig:proof}
\end{figure}

In turn, $\Omega_\lambda$ can be characterized, in terms of  $N^\delta(x)$ and $C^\delta(x)$, as follows:
\begin{align}
\ov{\Omega_\lambda} = \bigcup_{x\in D^\delta} N^\delta(x), \qquad
\partial\Omega_\lambda = \bigcup_{x\in\partial D^\delta} C^\delta(x).\label{eq:ChOmegaLam}
\end{align}

We remark that $\{N_{D^\delta}(x) \sepa x\in D^\delta\}$ is a family of disjoint sets since  any point $y\in N_{D^\delta}(x)$ satisfies $\proj_\delta(y)=x$.
Thus $\{N^\delta(x): x\in D^\delta\}$ determines a partition of $\ov{\Omega_\lambda}$.
On the other hand, when $x$ is a regular point of $\partial D^\delta$, i.e. $x\in\partial_r D^\delta$, then the normal cone $N_{D^\delta}(x)$  reduces to only one direction and $C^\delta(x)=\{x^*\}$ where $x^*$ is the unique projection of $x$ on $D^c$. Thus $x\notin \Lambda_D$ and $x^*$ belongs to $\partial D\cap\partial\Omega_\l$. If $x\in \partial_s D^\delta$, the arc $C^\delta(x)$  given by
\begin{align*}
C^\delta(x)= x+\left(\operatorname{cone}\{ \nu_{D^\delta}^-(x), \nu_{D^\delta}^+(x)\}\cap \partial {B(0,\delta)} \right)
\end{align*}
is determined by the angular interval $(\varphi_-^\delta(x),\varphi_+^\delta(x))$, which by the convexity of $D^\delta$ must be less than $\pi$.
Accordingly this arc is a connected piece  of the free boundary $\partial \O_\l \cap D$. Conversely, every $y\in \partial\Omega_\lambda$ can be decomposed as $y = \proj_\delta (y) + (y-\proj_\delta (y))$ with $|y- \proj_\delta (y)| =\delta$, hence $y\in C^\delta(\proj_\delta (y))$. This confirms the second equality in \eqref{eq:ChOmegaLam} where  $\{C^\delta(x) \sepa x\in\partial D^\delta\}$ is a partition of $\partial\Omega_\lambda$. 

\medskip
To shorten the notation, for each $x_j\in \partial_s D^\delta$ where $j\in J$, we set 
\begin{align}
 N^\delta_j :=& \; N^\delta(x_j), & C^\delta_j :=& \; C^\delta(x_j), \label{set:Cj}\\
 \nu_j := &\; \frac{\nu_{D^\delta}^+(x_j)+ \nu_{D^\delta}^-(x_j)}{2}, & \varphi_j:=&\; \frac{\varphi_+^\delta(x_j) - \varphi_-^\delta(x_j)}{2}.\label{set:nuj}
\end{align}
We point out that $\partial\Omega_\lambda\cap D =\bigcup_{j \in J} C^\delta_j$, where each  $C^\delta_j$ is an arc of radius $\delta$ determined  by a triple 
$(x_j,\varphi_j,\nu_j)\in \partial_s D^\delta\times(0,\pi/2)\times S^1$,  which represent respectively the center, the  angle, and an oriented unit vector 
(see the right hand side of Figure \ref{Fig:EomegaLambda}).

Now, let us introduce the regions where we want to construct the vector field $q_\lambda$,
\begin{gather}
\tinm{for} x\in D^\delta,\quad M^\delta_0 (x) = x+ \aco{ p\in N_{D^\delta}(x) \sepa \ps{ p, \nu_{D^\delta}^+(x) } \le \delta,\;  \ps{ p, \nu_{D^\delta}^-(x) } \le \delta },\label{set:M}\\
\Sigma_\lambda:= \bigcup_{x\in D^\delta} M^\delta_0(x), \qquad \qquad \Delta_\lambda:= \Sigma_\lambda\setminus\ov{\Omega_\lambda}. \label{set:sigmal}
\end{gather}

\begin{figure}[H]
\centering
\includegraphics[scale=0.97]{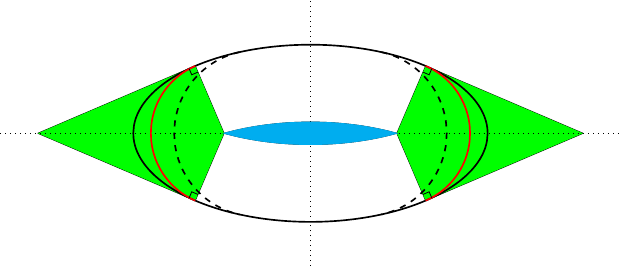}
\includegraphics[scale=1]{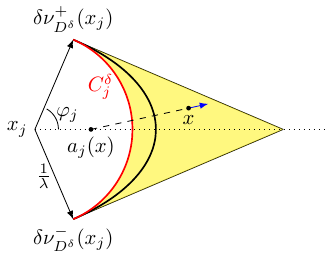}
\caption{$\Sigma_\l$ in green and  $\Delta_\l=\Sigma_\l\setminus\ov{\O_\l}$ in yellow.}\label{Fig:EomegaLambda}
\end{figure}

Figure \ref{Fig:EomegaLambda} gives the flavour  of the construction when $D$ is an ellipse domain. The thick black curve stands for the boundary of $D$ while the dashed curve represents the boundary of the Cheeger set $\Omega_{h_D}$ of $D$. The cyan region is the set $D^\delta$ whose $\delta$-enlargement gives $\Omega_\lambda$. The boundary of $\Omega_\lambda$ appears in red color. The sets $M^\delta_0(x_j)$ are in green background. The region in light yellow background is a component of $\Delta_\lambda$. $C^\delta_j$ are red arcs of radius $\delta$, centered at $x_j$, determined by angle $\varphi_j$ and the oriented unit vector $\nu_j$. $\{C^\delta_j\}$ are the boundaries of $\Omega_\lambda$ inside $D$, i.e. $\partial\Omega_\lambda\cap D$.

The set $\ov D$ is contained in $\Sigma_\lambda$. In fact, for every $y\in \ov D$, let $x= \proj_\delta (y)$ and $p=y-x$. Since $|x-(x+\delta \nu_{D^\delta}^\pm(x))|=\delta=d(x,D^c)$ and  $x+\delta \nu_{D^\delta}^\pm(x)\in \partial D$, we have that $x+\delta \nu_{D^\delta}^\pm(x)$ are projections of $x$ on $D^c$. As $\ov D$ is convex, these projections implies that for all $z\in \ov D$,
\begin{align*}
\ps{ x- (x+\delta \nu_{D^\delta}^-(x)), z - (x+\delta \nu_{D^\delta}^-(x)) } \ge 0,\\
\ps{ x- (x+\delta \nu_{D^\delta}^+(x)), z - (x+\delta \nu_{D^\delta}^+(x)) } \ge 0.
\end{align*}
Therefore, for $z=y$, we obtain 
$\langle p, \nu_{D^\delta}^\pm(x) \rangle \le \delta$. Hence, by definition \eqref{set:M}, $y=x+p\in M^\delta_0(x)$. In particular, when $x\in \partial_r D^\delta$, the left and right limits are the same, i.e. $\nu_{D^\delta}^- (x) = \nu_{D^\delta}^+(x) =: \nu_{D^\delta} (x)$, while $N^\delta(x)$ coincides with $M^\delta_0(x)$ and they are folded up to be a segment. At that moment, $p$ and $\nu_{D^\delta} (x)$ are co-linear, then, $\langle p, \nu_{D^\delta} (x)\rangle = |p| \le \delta$. This is to say  that $y\in N^\delta(x) = M^\delta_0(x)$. Of course, when $y\in D^\delta$ or $y= \proj_\delta (y)$, the associated cones degenerate and  shrink to a point. We get $y= \proj_\delta (y) =  N^\delta(y) = M^\delta_0(y)$.

 We notice that since the angle of $N_{D^\delta}(x)$ is always less than $\pi$, $M^\delta_0(x)$ is bounded for every $x\in D^\delta$. $\Sigma_\lambda$ is then bounded. It is clear that $\{M^\delta_0(x) \sepa x\in D^\delta\}$ is a partition of $\Sigma_\lambda$ and hence, $\Delta_\lambda$ admits a decomposition,
\begin{align}
\Delta_\lambda = \bigcup_{x\in D^\delta} M^\delta_0(x) \setminus\ov \Omega_\lambda = \bigcup_{x\in \partial_s  D^\delta} M^\delta_0(x) \setminus \ov{\Omega_\lambda}.
\end{align}
For short, we set
\begin{align}
\Delta_\lambda ^j=  M^\delta_0(x_j) \setminus\ov \Omega_\lambda \tinm{for some} x_j\in\partial_s D^\delta, \qquad \tinm{and}\quad \Delta_\lambda = \bigcup_{j\in J} \Delta_\lambda ^j.
\end{align}
$ \Delta_\lambda$ has  at most countable many disjoint components. Figures \ref{Fig:EomegaLambda} and \ref{Fig:omegaLambda} are illustrating examples. In Figure \ref{Fig:EomegaLambda}, $\Sigma_\lambda$ strictly contains $D$ whereas $\Sigma_\lambda=D$  in case of Figure \ref{Fig:omegaLambda}.

We now explicitly construct the vector field $q_\lambda$  in $\Delta_\lambda$ satisfying conditions \eqref{qlambda}. In each component $\Delta^j_\lambda$, we set $x=(s,t)$ and
\begin{align}
\Delta^j_\lambda \ni (s,t)\mapsto q_1^j(s,t):= (s - a_j(s,t)\nu_j^s, t - a_j(s,t)\nu_j^t),\label{def:q1ext}
\end{align}
where $\nu_j=(\nu_j^s,\nu_j^t)$ is the oriented unit vector defined $C^\delta_j$ and $a_j(s,t)\ge 0$ such that $\Div q_1^j = 1$. In fact, $q_1^j$ is the unit normal of the ball of radius 1 centered at point $a_j\nu_j$,
\begin{align}
[s - a_j(s,t)\nu_j^s]^2 + [t - a_j(s,t)\nu_j^t]^2 = 1. \label{expansionq1}
\end{align}
We observe also that 
\begin{align}
\Div q_1^j(s,t) = 1 \Longleftrightarrow \partial_s a_j(s,t) \nu_j^s + \partial_t a_j(s,t)\nu_j^t = 1. \label{expansionq2}
\end{align}
From equation \eqref{expansionq1}, we can find out explicitly $a_j$ in function of $(s,t)$, and in such a way, \eqref{expansionq2} is fulfilled,
\begin{align}
a_j(s,t)= s\nu_j^s + t\nu_j^t - \sqrt{1 - (s\nu_j^t - t\nu_j^s)^2}.\label{def:q1trans}
\end{align}
Therefore, $q_\lambda(s,t)=q_\lambda^j(s,t):= q_1^j(\lambda s,\lambda t)$ in $\Delta^j_\lambda$ as we expected. This completes the proof.
\end{proof}

\bigskip

\begin{example}\label{qexpli}
We precise here a calibrating field $\ov q$ when $D$ is  a square or  an ellipsoid.

\medskip
\begin{itemize}[leftmargin=4mm]
\item[(a)] In  case $D=[-1/2,1/2]^2$,  $\Sigma_\lambda$ coincides with $D$, see Figure \ref{Fig:omegaLambda}. The boundary of $D^\delta$ has 4 singular points and $\partial\Omega_\lambda\cap D = \cup_{j=1}^4 C^\delta_j$. The oriented vectors of $C^\delta_j$ are $(\pm 1/\sqrt{2},\pm 1/\sqrt{2})$. Take $\nu_1 = ( 1/\sqrt{2}, 1/\sqrt{2})$ for example, it is easy to  explicit  $q_\lambda$ in $\Delta_\lambda^1$. Thanks to \eqref{def:q1trans} and \eqref{def:q1ext}, we have, for $(s,t)\in \Delta^1_\lambda $,
\begin{align*}
a_1(s,t) = \frac{s+t}{\sqrt{2}} - \sqrt{1-\frac{(s-t)^2}{2}},
\qquad
q_1^1(s,t)= \left(s - \frac{a_1(s,t)}{\sqrt{2}}, t - \frac{a_1(s,t)}{\sqrt{2}} \right).
\end{align*}
Then, the wished construction of $q_\lambda$ in $\Delta_\l^1$ is given by  $q_\lambda(s,t)=q_\lambda^1(s,t)= q_1^1(\lambda s,\lambda t)$. 
By using symmetries, the expression of  $q_\lambda$ can be deduced in the other components of $\Delta_\lambda$. On the other hand, the field  $\ov q$ 
 is given in $\O_\l\setminus \O_{h_D}$ by $\ov q=-\frac{\nabla \rho}{|\nabla \rho|}$ which is described in Example \ref{ExRho}.

\item[(b)] In case $D$ is given by an ellipse of standard form, see Figure \ref{Fig:EomegaLambda}, $D$ is strictly contained in $\Sigma_\lambda$. $\Delta_\lambda$ now has two components and the boundary of $\Omega_\lambda$ inside $D$ is the union of arcs $C^\delta_1$ and $C^\delta_2$ whose oriented vectors are $(\pm 1,0)$. Take $\nu_1=(1,0)$ for example to construct $q_\lambda$ in $\Delta^1_\lambda$, we get, for $(s,t)\in \Delta^1_\lambda $,
\begin{align*}
a_1(s,t) = s - \sqrt{1-t^2}, \qquad q_1^1(s,t)= \left(\sqrt{1-t^2}, t \right).
\end{align*}
Therefore, we obtain $q_\lambda(s,t)$ by scaling $q_1^1(s,t)$, i.e. $q_\lambda(s,t)=q_\lambda^1(s,t)= q_1^1(\lambda s,\lambda t)$.
\end{itemize}
\end{example}

In summary, the vector field $\ov q$ can be built with $q_{h_D}$ on $\Omega_{h_D}$, with $q_\rho$ on $\Omega_\lambda\setminus \ov{\Omega_{h_D}}$ by means of cut-locus potential $\rho$, then glued with $q_\lambda$ on $D\setminus \ov{\Omega_\lambda}$  so that we can obtain a calibrating field $\ov q$ for $\Omega_\lambda$. We remark  that the  construction of the calibrating field $\ov q$ can be done in the  domain $\Sigma_\lambda$ (see \eqref{set:sigmal}) which, in general is larger than $D$.

\appendix
\section{Some techical lemmas in $\R^2$}\label{A} 

The following preparatory lemmas will be used in Section \ref{scutlocus}.
We recall the definitions of functions $\gamma$ and $\zeta$ in \eqref{ndis} and also of the cut-locus  $\ov\Lambda$ in Section \ref{sec_cutlocus}.

\begin{lemma}\label{stmonotone}
Let $D$ be a bounded convex set in $\RR^N$. Then
\begin{itemize}[leftmargin=5mm]
\item[(i)] the function $(x,\delta) \to   d(x, D^\delta)$ is continuous on $\ov D  \times [0,R)$. 
 
 \item [(ii)] Given $x\in D$ such that $d(x,D^\delta)\le \delta$ for some $\delta>0$. Then, for every $\delta'$ such that $0<\delta' < \delta$, we have $d(x,D^{\delta'}) < \delta'$.
\end{itemize}
\end{lemma}

\begin{proof}  \ Let us prove (i). Since the functions $d(\cdot,D^\delta)$ are $1$-Lipschitz, it is enough to show that, for any fixed $x\in \ov D$,
the function  $t \to d(x,D^t)$ is continuous on $[0,R)$. 
Let $t_n, t\in [0,R_D)$ such that $t_n \to t$. It exist a unique $y_n \in \ov {D^{t_n}}$ such that $|x-y_n|= d(x, D^{t_n}$. Up to a subsequence, we can assume that
$y_n \to y_*$ for some $y_*\in \ov D$. Then $d(y_n, D^c)\ge  t_n$ implies that $d(y_*, D^c)\ge t$ while $|x-y_n|\to |x-y_*|$. 
It follows that $y_*\in \ov { D^t}$ \footnote{Here we use the fact that $\ov {D^t}=\{y\in D\, :\, d(y,D^c)\ge t\}$. To show the non trivial inclusion of the second set in the first one, we consider  $y_n=  (1-\frac1{n}) y + \frac1{n} z$ where $d(y,D^c)= t$  and $z\in D^t$. Then by the concavity of the function $d(\cdot,D^c)$ on $D$,
we have that $d(y_n,D^c) >t$  while  $y_n\to y$ whence $y\in \ov { D^t}$. }, hence $|x-y_*|= \lim_n d(x, D^{t_n})  \ge d(x,D^t)$. 
In the opposite direction, let  $\e>0$ and $x_\e \in D^t$ such that $|x-x_\e|\le d(x, D^t)+\e$.  Since $d(x_\e, D^c) >t$, we have $x_\e\in D^t_n$ for large $n$
so that $ \limsup_n d(x, D^{t_n}) \le |x-x_\e|\le  d(x, D^t)+\e$. By sending $\e\to 0$, we conclude that $\limsup_n d(x, D^{t_n})\le d(x,D^t)$.
The wished continuity property is proved.

\medskip 
Let us prove now the assertion (ii). 
Without any loss of generality, we can assume that $x\notin D^{\delta'}$, hence $x\notin \ov{D^\delta}$, since $\ov{D^\delta}\subset D^{\delta'}$). Let us denote
$x_\delta= \proj_\delta(x)$ and   $x_{\delta'}= \proj_{\delta'}(x')$. As $d(x,D^\delta) \le \delta$,  we have  $d(x,x_\delta)\le \delta$ while
$B(x_\delta,\delta)\subset D$. 
We claim that we can always find out a ball $B(z,\delta')$ such that: 
\begin{equation}\label{claim:z}
\ov{B(z,\delta')} \subset B(x_\delta,\delta)\ ,\quad d(x,z) \le \delta'\, . 
\end{equation}
 If the claim is true, then $z\in D^{\delta'}$ since the first inclusion in \eqref{claim:z} implies that $\ov{B(z,\delta')} \subset D$. 
 On the other hand, $x_{\delta'}$ is the unique point of $\partial D^{\delta'}$ such that   $d(x,D^{\delta'})= |x-x_{\delta'}|$. The wished strict inequality
 follows, namely:
 \begin{align*}
d(x,D^{\delta'}) = |x - x_{\delta'}| < |x - z| \le \delta'.
\end{align*}
Let us prove \eqref{claim:z}. If $d(x,D^\delta) \le \delta'$, we can take $z=x_\delta$ so that $x\in \ov{B(z,\delta')}\subset B(x_\delta,\delta)$.

Let us now consider the case where   $d(x,D^\delta)>\delta'$. Then we choose  $z$ on the segment $[x,x_\delta]$ as follows:
\begin{align*}
z=\left(1-\frac{\delta'}{\delta}\right) x +\frac{\delta'}{\delta} x_\delta.
\end{align*}
Since $x\notin \ov{D^\delta}$, we have $d(x,x_\delta)< \delta$. Therefore
$ d(x,z) < \delta'$ and $d(x_\delta,z) <  \delta-\delta' .$ 
From the previous inequality  and by using the triangle inequality, we deduce  the inclusion $\ov{B(z,\delta')} \subset B(x_\delta,\delta)$. 
It follows that $z$ satisfies \eqref{claim:z}. \end{proof}

\begin{lemma}
\label{Mdd}
Recalling the definition \eqref{set:M} for the set-valued function $M_0^\delta(x)$, we have:
\begin{itemize}[leftmargin=5mm]
\item[(i)] Let $\delta, \delta'\in [0,R_D]$ such that $ \delta' < \delta$. Then
 $D^\delta = \aco{ x\in D^{\delta'} \sepa B(x,\delta-\delta')\subset D_{\delta'}}$. 
\medskip
\item[(ii)] Let $x\in \partial_r D^\delta$ and $z\in \partial D\cap M_0^\delta(x)$ (see \eqref{set:M}). Then, for every $y\in M_0^\delta(x)$, we have  $|y - z| = d(y,D^c)$.
In particular, $|x - z| =d(x,D^c)=\delta$.

\medskip
\item[(iii)] For each $\delta' <\delta$, it holds
\begin{align*}
\partial_s D^{\delta'} \subset \bigcup_{x\in\partial_s D^\delta} M_0^\delta(x).
\end{align*}
As consequence, for every $x\in \ov D$, if $\proj_\delta(x)$ is in $\partial_r D^\delta$ then $\proj_{\delta'}(x)$ belongs to $\partial_r D^{\delta'}$ for all $\delta'<\delta$.
\end{itemize}
\end{lemma}

\begin{proof}
(i) First we  show that $D^\delta\subset E$ where 
 $$E:=\aco{ x\in D^{\delta'} \sepa d(x, (D^{\delta'})^c) > \delta-\delta' } =\aco{ x\in D^{\delta'} \sepa B(x,\delta-\delta') \subset D^{\delta'} }.$$
For every $x\in D^\delta$, $B(x,\delta)\subset D$ implies that
\begin{align*}
\delta < d(x,D^c) \le d(x,\partial D^{\delta'}) + d(\partial D^{\delta'}, D^c) = d(x,\partial D^{\delta'}) +\delta'.
\end{align*} 
That means $d(x,(D^{\delta'})^c) = d(x,\partial D^{\delta'}) > \delta-\delta'$. It is to say that $x$ is in $E$.This $D^\delta \subset E$.

Conversely, for every $x\in E$, the inclusion $B(x,\delta-\delta')\subset D^{\delta'}$ implies that $ d(\partial B(x,\delta-\delta'), D^c) >\delta'$. Thus, we obtain
\begin{align*}
d(x, D^c)
&= d(x,\partial B(x, \delta-\delta')) + d(\partial B(x,\delta-\delta'), D^c) \\
& > (\delta-\delta') + \delta' =\delta.
\end{align*}
This shows that $x\in D^{\delta}$, whence $E\subset D^{\delta}$. 

\medskip(ii) Given $x\in \partial_r D^{\delta}$ and $z\in\partial D\cap M_0^\delta(x)$, we have $z=x+p$ with $|p|\le \delta$. It follows that
\begin{align*}
\delta =d(x, D^c) \le |x - z| \le \delta.
\end{align*}
Or, $d(x,z)=d(x,D^c)=\delta$.

For every $y\in M_0^\delta(x)$, we will prove that $|y - z|=d(y, D^c)$. Suppose that $d(y,D^c) < |y - z|$ and $d(y, D^c)= |y - \ov z|$ for some $\ov z\in \partial D$, $\ov z\neq z$. Then, $|y - \ov z| < |y - z|$. Recall that as $x\in\partial_r D^\delta$, $M_0^\delta(x)$ is a segment joining $x$ and $z$. We have
\begin{align*}
|x - \ov z| &\le |x - y| + |y - \ov z| \\
&< |x - y| + |y - z| = \delta
\end{align*}
while $|x - \ov z| \ge d(x, D^c) =\delta$. This gives a contradiction. So, $d(y,D^c)= |y - z|$.

\medskip (iii) It is equivalent to prove that for every $\delta' <\delta$, $y\in\partial_s D^{\delta'}$ implies that $\proj_\delta (y) \in \partial_s D^\delta$. Suppose that $y\in \partial_s D^{\delta'}$ and $\proj_\delta (y)\in \partial_r D^\delta$. From (i) and (ii), we derive that $|y-\proj_\delta (y)|=\delta-\delta'$. As $\proj_\delta (y)\in \partial_r D^\delta$, the ball $B(\proj_\delta (y),\delta)$ then touches the boundary of $D$ at a unique point called $z$. Let $w$ be the intersection of segment $M_0^\delta(\proj_\delta (y))=[\proj_\delta (y),z]$ and $\partial D^{\delta'}$. By using (ii), $d(w,D^c)=\delta'$. The ball $B(w, \delta')$ contained in $B(\proj_\delta (y),\delta)$ touches $\partial D$ at and only at $z$. In other words, $w$ is in $\partial_r D^{\delta'}$. Besides, both $y$ and $w$ are in $M_0^\delta(\proj_\delta (y))$. It is easy to see that they coincide. We conclude that $y$ belongs to $\partial_r D^{\delta'}$, a contradiction. The proof is complete.
\end{proof}

\begin{lemma}
The following assertions hold true:\label{lem:mod1}
\begin{itemize}
\item[(i)] Given $x\in\partial D^\delta$ and $x^*\in \partial D$ be such that $|x-x^*|=d(x,D^c)$. Then, for every $\delta'\le \delta$, $\proj_{\delta'} (x^*)$ belongs to the open segment $]x,x^*[$, that is: 
\begin{align*}
\exists t\in (0,1) \sepa \proj_{\delta'} (x^*)= (1-t) x+ t x^*.
\end{align*}

\item[(ii)] Let  $y\in M^\delta_0(x)$ such that $x=\proj_\delta (y)$. Then   $\proj_{\delta'} (y) \in M^\delta_0 (x)$ for every $\delta'<\delta$.
\end{itemize}
\end{lemma}

\begin{proof}
(i)
For each $z\in[x,x^*]$, $z$ can be parametrized as
\begin{align*}
z(t):= (1-t)x^* + tx.
\end{align*}
If we take $z^*=z(\delta'/\delta)$ then $z^*\in \partial D^{\delta'}$. In fact, it holds $d(z^*,D^c)=\delta'$ since
\begin{align*}
\delta' = |z^* - x^*| \ge d(z^*,D^c) \ge d(z^*,\partial B(z^*,|z^*-x^*|))=\delta'.
\end{align*}
Besides, we have
\begin{align*}
\delta' = | x^* - z^*| \ge \inf_{y\in D^{\delta'}} |x^* - y| \ge \delta'.
\end{align*}
It turns out that $\proj_{\delta'} (x^*)=z^*$. 

%
%

(ii) 
Suppose that $y^*:= \proj_{\delta'} (y) \notin M^\delta_0 (x)$ (see Figure \ref{fig:proof} for illustration). Without loss of generality, we can assume that the segment $[y,y^*]$ intersects $[x,x^+]$ at $z$, where $x^+\in \partial D$ such that $|x-x^+|=d(x,D^c)$. Let $z^*$ be the intersection of the segment $[x,x^+]$ and $\partial D^{\delta'}$. By (i), we get that $\proj_{\delta'} (x^+)=z^*$. So, the segment $[z^*,x^+]$ is contained in $ M^{\delta'}_0 (z^*)$. 
Therefore, for every $z\in [z^*,x^+]$, $z$ admits $z^*$ as its unique projection on $D^{\delta'}$. It is then clearly that
\begin{align*}
|y-z^*|\le |y-z|+|z-z^*| \le |y-z| + |z-y^*|=|y-y^*|
\end{align*}
gives a contradiction to the fact that $y^*:= \proj_{\delta'} (y)$.
\end{proof}

\begin{lemma}\label{alpha}
Let $\alpha(x,\delta)$ be a function defined on $\ov D\times [0,R]$ by
\begin{align}\label{def:alpha}
\alpha(x,\delta) := d(x,D^\delta) - \delta .
\end{align}
Then, for every  $x\in D\setminus\ov{\Omega_{\frac{1}{R}}}$, the function $\alpha(x,\cdot)$ is continuous on $[0,R]$ and exhibits the following behavior:
\begin{align*}
\alpha(x,\delta)= 
\begin{cases}
 -\delta \qquad &\tinm{if} 0 < \delta \le d(x,D^c) \\
 -d(x,D^c) & \tinm{if} x\notin \Lambda_D \, \tinm{and}\, d(x,D^c) < \delta < \gamma(x)\\
\tinm{is strictly increasing} &\tinm{on } [\gamma(x),R] .
\end{cases}
\end{align*}
Therefore, if $x\in D$, $\alpha(x,\cdot)$ vanishes at the unique positive  $\delta:=\rho(x)$ where $\gamma(x)< \rho(x)< R$.

\medskip
On the other hand, if $x\in \partial D$, then $\gamma(x)=\tau(x)$ and $\alpha(x,\cdot)$ vanishes on $[0,\tau(x)]$ while $\alpha(x,\cdot)>0$ on the possibly empty
inerval $(\tau(x), R]$.
\end{lemma}
\begin{proof}
The continuity statement results from the assertion (i) of Lemma \ref{stmonotone}.
 If $0<\delta \le d(x,D^c)$, then $x$ belongs to $\ov{D^\delta}$ and  $d(x,D^\delta)= 0$ implies that  $\alpha(x,\delta)=-\delta$.

\medskip
Next, we assume that  $x\notin \Lambda$ and  we prove that $\alpha(x,\delta) = -d(x,D^c)$ if $\delta$ belongs to the interval $(d(x,D^c),\gamma(x))$. 
Note that, by the definition  \eqref{ndis}, it holds  $\eta(x)=0$ if $x\in \Lambda_D$, making the latter interval is empty. Then it exits $\ov x\in \ov\Lambda$ 
and $x^*\in \partial D$ such that
\begin{align*}
\ov x= x +\zeta(x) \nabla d(x,D^c)\ ,\quad |\ov x - x^*| = d(\ov x,D^c)= \gamma(x) .
\end{align*}
Since we assumed that  $d(x,D^c)< \delta <\gamma(x)$, $D^\delta$ meets the segment $[\ov x, x^*]$ at a unique point $w$ such that  $w\in [\ov x,x^*]$ and $w \in \partial_r D^\delta$. By Lemma \ref{lem:mod1} (i), we deduce that $w = \proj_\delta (x^*)= \proj_\delta (x)$.
As a consequence, keeping in mind that $|x-x^*| = d(x,D^c)$ (see Lemma \ref{Mdd} (ii)), we obtain
\begin{align*}
d(x, D^\delta) = |x- w|  =  |x^*-w| - |x^*-x| = \delta - d(x,D^c).
\end{align*}
So, we conclude that $\alpha(x,\delta) = d(x,D^\delta) -\delta = -d(x,D^c)$ as claimed. 

\medskip
Finally, let us show that  the continuous function $\alpha(x,\cdot)$ is strictly increasing in $[\gamma(x), R]$. If it is the case, then as  $\alpha(x,\gamma(x))= -d(x,D^c) < 0$
and $\alpha(x,R)= d(x,D^R) - R > 0$ (since $x\notin \ov U_D$), we deduce 
 the existence of a unique $\delta\in (\gamma(x), R)$ such that $\alpha(x,\delta) =0$ as claimed.

Let $\delta$ and $\delta'$ such that $\gamma(x) <\delta'< \delta<R$. Then, since $\gamma(x)\ge d(x,D^c)$, we have  $d(x,D^c)<\delta'$ , hence   the ball $B(x, |x- \proj_{\delta'}(x)|)$ is contained in $(D^{\delta'})^c$ . It follows that
\begin{align*}
d(x,D^\delta)- d(x, D^{\delta'}) 
&= |x- \proj_\delta (x)| -|x- \proj_{\delta'} (x)| \\
&= d ( \proj_\delta (x), B(x, |x- \proj_{\delta'} (x)|) ) \\
&\ge d( \proj_\delta (x), \partial D^{\delta'})\\
&= \delta- \delta'.
\end{align*}
The third line inequality  becomes an equality  if and only if
\begin{align}
\proj_{\delta'} (\proj_\delta (x)) = \proj_{\delta'} (x) = \proj_{\ov {B(x,|x- \proj_{\delta'} (x)|)} } (\proj_\delta (x))  .\label{alp:mono2}
\end{align}
In this case $\Pi_{\delta'}(x)$ belongs to to the segment  $[ \proj_\delta (x),x]$ and  $\proj_\delta (x)\in \partial_r D^\delta$. By the asertion (iii) of Lemma \ref{Mdd}, it follows that $\proj_{\delta'} (x)\in \partial_r D^{\delta'}$ while, by assertion (ii),  $M_0^\delta(\proj_\delta (x))$ contains  the segment
$S:= \{ x +t_\delta \nabla d(x,D^c)\:\ t\in [0, \zeta(x)\}$. In paticular $\proj_\delta (x)$ can be rewritten as
\begin{align*}
\proj_\delta (x) = x +t_\delta \nabla d(x,D^c),
\end{align*}
for some $t_\delta>0$. Since $\gamma(x) < \delta$, we get $\zeta(x) < t_{\delta}$, or equivalently that  $x+\zeta(x) \nabla d(x,D^c) $ belongs to $M_0^\delta(\proj_\delta (x))$. This gives a contradiction to the fact that $x+\zeta(x)  \nabla d(x,D^c) $ is a singular point of $d(\cdot,D^c)$. So, we conclude that the relation \eqref{alp:mono2} never occurs for $\delta'>\gamma(x)$. It follows  $\alpha(x,\delta)>\alpha((x,\delta')$, whence the claimed strict monotony
property on $[\gamma(x), R]$.

Let us finally conclude with the case where $x\in \partial D$. Then $\gamma(x)=\tau(x)$ and clearly $\alpha(x,\cdot)$ satisfies the required properties 
 since $\tau(x)= \max \{t\ge 0: P_{\partial D}(x- t \, \nu_D(x))=x\}$ if $x\in \partial_r D$ and $\tau(x)=0$ if $x\in \partial_s D$.

\end{proof}

\begin{lemma}\label{access}
Let $D$ be a convex  domain  of $\RR^2$. For every $x\in \partial D$, we define
\begin{align*}
k_{\partial D}(x):= \frac{1+ \nu_D^+(x)\cdot\nu_D^-(x)}{2}.
\end{align*}
Then, we have
\begin{itemize}
\item[(i)] for every $x\in\partial D$, $0<k_{\partial D}(x) \le 1$;
\item[(ii)] $k_{\partial D}(x) = 1$ for every $x\in \partial_r D$;
\item[(iii)] $\forall\e \in(0,1)$, the set  $\aco{x \sepa k_{\partial D}(x) <\e }$ is finite.
\end{itemize}
\end{lemma}

\begin{proof}
We recall that the normal cone of $D$ at $x$ is given by
\begin{align*}
N_{D}(x):=\aco{ a \nu_D^+(x) + b \nu_D^-(x) \sepa a,b \in \RR_+ }.
\end{align*}
For every $x\in \partial D$, we denote by $\varphi(x)$ the angle
\begin{align*}
\varphi(x):= \frac{1}{2}\angle(\nu_D^-(x),\nu_D^+(x)),
\end{align*}
and by $T_D(x)$ the tangent cone of $D$ at $x$
\begin{align*}
T_D(x):= & \cl \aco{ s(y-x) \sepa y\in D,\; s\ge 0 }\\
=& \aco{a T_D^+(x) + b T_D^-(x) \sepa a,b \in \RR_+},
\end{align*}
where $T_D^\pm(x)$ are the left and right tangent unit vectors of $D$ at $x$. We denote by $\psi(x)$ the angle 
\begin{align*}
\psi(x):= \frac{1}{2}\angle (T_D^-(x),T_D^+(x)).
\end{align*} 
Since tangent and normal cones are polar each other, it holds $\varphi(x)+\psi(x) = \frac{\pi}{2}$ for every $x\in \partial D$.

\noindent
\emph{Proof of (i):} clearly, as $D$ is convex, we have 
$ 0\le \varphi(x)< \frac{\pi}{2}$ and  $0< \psi(x)\le \frac{\pi}{2}$.
Hence, for all $x\in\partial D$, it holds $0<k_{\partial D}(x) \le 1$ since
\begin{align*}
k_{\partial D}(x)=\frac{1+\cos 2\varphi(x) }{2} =\cos^2\varphi(x) = \sin^2 \psi(x).
\end{align*}

\noindent
\emph{Proof of (ii):}  As $D$ is convex, its boundary $\partial D$ admits at most countably many singular points. On the regular part $\partial_r D=\partial D\setminus \partial_s D$, we get $\nu_D^-(x)=\nu_D^+(x)$, i.e. $\varphi(x)=0$. Then, $k_{\partial D}(x)=1$.

\medskip
\noindent
\emph{Proof of (iii):} Let $\e$ such that $0<\e<1$ and denote
\begin{align*}
E_\e:=\{ x\in\partial D \sepa k_{\partial D}(x) <\e \}\quad,\quad N_\e:=\#(E_\e) .
\end{align*}
  The function $h(\varphi)$ defined by
\begin{align*}
h(\varphi):=\frac{1+\cos 2\varphi}{2}, \quad \tinm{for} \varphi\in [0,\frac{\pi}{2}).
\end{align*}
 is non increasing in $[0,\frac{\pi}{2})$, while  $h(\varphi(x))=k_{\partial D}(x)$ for every $x\in \partial D$.
There always exists $\ \varphi_e\in (0,\pi/2)$ such that $h(\varphi_\e)=\e$. We set
\begin{align*}
E_\e:=\{ x\in\partial D \sepa k_{\partial D}(x) <\e \}.
\end{align*}
By the monotonicity of $h$, we deduce that  $\varphi(x) > \ov \varphi$ for every $x\in E_\e$.
It follows that
\begin{align*}
N_\e\, \ov\varphi  < \sum_{x\in E}\varphi(x) \le \sum_{x\in\partial D}\varphi(x) \le 2\pi ,
\end{align*}
whence $N_\e \le \frac{2\pi}{ \ov\varphi} <+\infty$.
\end{proof}

\end{document}